\documentclass[onefignum,onetabnum]{siamart190516}

\usepackage[numbers,sort&compress]{natbib}
\usepackage{graphicx,latexsym}
\usepackage{amssymb,amsmath}
\usepackage{verbatim}
\usepackage{bbm}
\usepackage{mathtools}
\usepackage{braket}
\usepackage{xcolor}
\usepackage{sansmath}
\usepackage{caption}
\usepackage{subfig}
\usepackage[utf8]{inputenc}
\usepackage{fancyhdr}
\usepackage{algpseudocode}
\usepackage{algorithm}
\usepackage[mathscr]{euscript}
\usepackage{hyperref}
\graphicspath{{./figures/}}

\newsiamremark{remark}{Remark}
\newsiamremark{hypothesis}{Hypothesis}
\crefname{hypothesis}{Hypothesis}{Hypotheses}
\newsiamthm{claim}{Claim}

\headers{Low-dimensional approximation to PDE solution manifold}{Shi Chen, Qin Li, Jianfeng Lu and Stephen J. Wright}

\title{Manifold learning and nonlinear homogenization\thanks{Submitted to the editors DATE\funding{The work of JL is supported in part by the National Science Foundation via grants DMS-1454939 and DMS-2012286. The work of SC, QL, and SW is supported in part by the National Science Foundation via grants CCF-1740707 and DMS-2023239. The work of SW is further supported in part by National Science Foundation via grant 1934612, Subcontract 8F-30039 from Argonne National Laboratory, and Award N660011824020 from the DARPA Lagrange Program. The work of SC and QL is further supported in part by Wisconsin Data Science Initiative and National Science Foundation via grant DMS-1750488.}}}

\author{Shi Chen\thanks{Mathematics Department, University of Wisconsin-Madison, Madison, WI 53706 ({\tt schen636@wisc.edu})}
	\and Qin Li\thanks{Mathematics Department and Discovery Institute, University of Wisconsin-Madison, Madison, WI 53706 ({\tt qinli@math.wisc.edu})}
	\and Jianfeng Lu\thanks{Department of Mathematics, Department of Physics, and Department of Chemistry, Duke University, Durham, NC 27708 ({\tt jianfeng@math.duke.edu}) }
	\and Stephen J. Wright\thanks{Computer Sciences Department, University of Wisconsin, Madison, WI 53706 ({\tt swright@cs.wisc.edu})}
}

\usepackage{amsopn}

\newcommand{\Kcal}{\mathcal{K}}

\newcommand{\Rbb}{\mathbb{R}}

\newcommand{\eps}{\epsilon}

\newcommand{\wt}[1]{\widetilde{#1}}
\newcommand{\ol}[1]{\overline{#1}}

\DeclareMathOperator*{\argmin}{argmin}

\newcommand{\algrule}[1][.2pt]{\hspace*{-.6in}\hrulefill}
\algdef{SE}[SUBALG]{Indent}{EndIndent}{}{\algorithmicend\ }%
\algtext*{Indent}
\algtext*{EndIndent}

\ifpdf
\hypersetup{
  pdftitle={Low-dimensional approximation to PDE solution manifold},
  pdfauthor={Shi Chen, Qin Li, Jianfeng Lu and Stephen J. Wright}
}
\fi

\begin{document}

\maketitle

\begin{abstract}
  We describe an efficient domain decomposition-based framework for
  nonlinear multiscale PDE problems. The framework is inspired by
  manifold learning techniques and exploits the tangent spaces spanned
  by the nearest neighbors to compress local solution manifolds. Our
  framework is applied to a semilinear elliptic equation with
  oscillatory media and a nonlinear radiative transfer equation; in
  both cases, significant improvements in efficacy are observed. This
  new method does not rely on detailed analytical understanding of the
  multiscale PDEs, such as their asymptotic limits, and thus is more
  versatile for general multiscale problems.
\end{abstract}

\begin{keywords}
  Nonlinear homogenization; multiscale problems; manifold learning;
  domain decomposition.
\end{keywords}

\begin{AMS}
  65N99
\end{AMS}

\section{Introduction} \label{sec:intro}
Homogenization is a body of theory and methods to study differential,
or differential-integro equations with rapidly oscillating
coefficients. It traces back to the famous work of
Bensoussan-Lions-Papanicolaou~\cite{BeLiPa:2011}, and builds on
several other important
developments~\cite{HoWu:1997,EWEn:2003,BaLi:2011,DuGo:2000,GoMe:2001,BaSaSe:1984,GoMe:2003,AbPuVo:2012}. Generally
speaking, the goal of homogenization is to derive asymptotic limiting
equations as accurate surrogates of the original equations that do not
have scale separations. The core technique is asymptotic analysis.

\subsection{Goal}
There are a number of famous examples that use homogenization
techniques, such as elliptic equations with rapidly oscillating
media~\cite{BeLiPa:2011}, Schr\"odinger equation with small rescaled
Planck constant~\cite{GeMaMaPo:1997}, the neutron transport equation
with small Knudsen number~\cite{BeLiPa:1979,GuWu:2017}, compressible
Euler equation with small Mach
number~\cite{KlMa:1981,KlMa:1982,Sc:1986}, and Boltzmann-type
equations in the fluid regime~\cite{BaGoLe:1991}. All these examples
have the form
\begin{equation}\label{eqn:generic}
\mathcal{N}^\eps u^\eps = f\,,
\end{equation}
where $\mathcal{N}^\eps$ is a partial differential operator that
depends explicitly on the small parameter $\eps$. The term $f$
on the right-hand side represents the external information---the source terms, the boundary conditions, the initial
conditions, and so on---which has no dependence on $\eps$. Due to
the $\eps$-dependence of $\mathcal{N}^\eps$, the PDE is rather stiff:
the solutions either exhibit high oscillations (such as the
Schr\"odinger equation with small value of the rescaled Planck
constant, or the elliptic equation with rough media), or present
boundary/initial layers within which solutions change rapidly (such as
the Knudsen layer in kinetic systems). The oscillations and layers
themselves usually do not carry any interesting physical information;
one is more interested in extracting physically meaningful quantities
from the solutions directly, with these details omitted. Thus, it is
important to evaluate the limiting behavior of \eqref{eqn:generic} as
$\eps\to 0$. There are two contrasting approaches in the literature
that enable this task: One is analytical and the other is numerical.

The analytical approach seeks the asymptotic limit of the PDE
\eqref{eqn:generic}, defined as follows:
\begin{equation}\label{eqn:generic_hom}
\mathcal{N}^\ast u^\ast = f\,.
\end{equation}
The term ``asymptotic limit'' refers to the fact that for any
reasonable $f$, in a certain space with a certain metric, we have
\begin{equation} \label{eq:epserr}
\|u^\eps-u^\ast\|\to 0 \quad \mbox{as $\eps \to 0$.}
\end{equation}
A classical way to derive this limit is to perform Hilbert expansion
in terms of $\eps$. Here, we define the ansatz
\[
u^\eps = u_0 +\eps u_1 + \eps^2 u_2+\dots
\]
and substitute into \eqref{eqn:generic}, then balance the two sides in
terms of $\eps$. Typically, at some level of the expansion, a closure
is performed to derive the effective operator $\mathcal{N}^\ast$. This
framework is highly effective and general; we will give explicit
examples in later sections.

On the numerical side, we look for cheap solvers that compute the
asymptotic limits. A typical requirement for classical numerical
solvers to be accurate is that the discretization has to resolve the
smallness of $\eps$. This can lead to high numerical and memory cost,
sometimes beyond reasonable computational resources. The focus of
``numerical homogenization" or ``asymptotic preserving" is thus to
design schemes that capture asymptotic limits of the solutions with
relaxed (and thus more efficient) discretization requirements. One
technique is to explore analytical results and translate them to the
discrete setting: The asymptotic limiting equations are derived first,
and then a ``macro solver'' for the limiting equation and a ``micro
solver'' that solves the original equation, are combined in some
way. This strategy has been applied to deal with the Boltzmann-type
equations, the Schr\"odinger equation, and the elliptic equations with
highly oscillatory media, under the name of designing ``asymptotic
preserving'' schemes, finding semi-classical limits, and performing
``numerical homogenization.'' There is a significant drawback of this
approach: The design of the numerical method is based completely on
analytical understanding, so numerical development necessarily lags
analytical progress. This fact significantly limits the role of
multiscale computation.

This observation motivates the question that we address in this paper.
Given a system of the form~\cref{eqn:generic}, knowing it has an
asymptotic limit~\cref{eqn:generic_hom} but not knowing the specific
form of this limit, can we design an efficient, accurate solver?

\subsection{Approach}

In this paper, we propose a numerical approach based on
``compression.''  Classical methods require the use of
$N_\eps\sim\frac{1}{\eps^\alpha}$ grid points to achieve accuracy and
stability in solving~\cref{eqn:generic}, for some power
$\alpha>0$. Note that $N_\eps$ blows up to infinity as $\eps \to
0$. By contrast, the limiting equation \eqref{eqn:generic_hom} is
independent of $\eps$, so we typically require only $N_\ast$ grid
points (a number that is independent of $\eps$) to solve this
system. Thus, the information carried in $N_\eps$ degrees of freedom
can potentially be ``compressed'' into $N_\ast$ degrees of freedom,
provided that we can tolerate an asymptotic error of order $\eps$ in
the solution (see \eqref{eq:epserr}).

How can we design an approach to solving \eqref{eqn:generic} that
exploits compression? Our roadmap consists three steps: (a) identify
the solution set that can be compressed; (b) compress the set into a
smaller effective solution set; (c) for a given new data point $f(x)$,
single out the solution from the effective set. We call the first two
steps the offline stage, and the last step the online stage.

For linear equations, this roadmap has been followed by several
authors
in~\cite{buhr2018randomized,ChLiLuWr:2018randomized,ChLiLuWr:2018,ChLiLuWr:2019}. When
the setup is linear, the solution set is a space, and thus information
is entirely coded in representative basis functions. These basis
functions can be found in the offline stage, and a Galerkin
formulation can then be used to identity the linear combination of the
basis for a given $f(x)$ in the online stage. To find the
representative basis functions, one can utilize the random sampling
technique developed for finding low rank structures of matrices
in~\cite{HaMaTr:2011}, where the authors proved that a few random
samples are able to reconstruct the low-rank column space with high
probability; see~\cite{ChLiLuWr:2018randomized}.

In this article, we develop the roadmap in the nonlinear setting. The
extension is not straightforward. Since the solution set is not a
space in the nonlinear setup, the notion of ``basis function'' does
not even exist. Instead, we seek an $N_\ast$-dimensional approximating
manifold in an $N_\eps$-dimensional space. For every given $f(x)$,
there is a corresponding numerical solution $u^\eps$ to the original
equation~\cref{eqn:generic} in the $N_\eps$ space. Within $\eps$
distance there exists its homogenized solution $u^\ast$ to the
limiting equation~\cref{eqn:generic_hom}. Since $u^\ast$ relies on
only $N_\ast$ degrees of freedom, as $f(x)$ varies, the variations of
$u^\ast$ form a manifold of dimension at most $N_\ast$.

By using this argument, we formulate the homogenization problem (in
the nonlinear setting) into a manifold-learning problem: Suppose we
can generate a few configurations of $f(x)$ and compute the associated
numerical solutions, can we learn to represent the solution manifold?
Further, given a completely new configuration of $f(x)$, can we
quickly identify the corresponding solution? These two questions are
addressed in the offline and online stages, respectively.

Many different approaches have been proposed for manifold learning
based on observed point clouds. They typically look for key features
that the points share, either locally (as in local linear embedding
(LLE)~\cite{RoSa:2000}, multi-scale
SVD~\cite{AlChMa:2012,LiMaRo:2017}, local tangent space
alignment~\cite{ZhZh:2004}), or globally (as in the use of heat
kernels~\cite{CoLaLeMaNaWaZu:2005,BrNi:2002}). The strategy we propose
here is not a direct application of any one of these ideas, but it
uses elements of the the local linear embedding and multi-scale SVD
approaches. Specifically, we seek local linear approximations to the
solution map, and cover the solution manifold with a number of these
tangent space ``patches.''

We define the solution map as follows:
\begin{equation} \label{eq:defS}
\mathcal{S}^\epsilon:\quad f\in\mathcal{X} \to u^\epsilon\in\mathcal{Y}\,.
\end{equation}
It maps the source term and initial/boundary conditions captured in
$f(x)$ to the solution of the equation \eqref{eqn:generic}.
To find the solution manifold, we randomly sample a large number of
configurations $f_i$ in $\mathcal{X}$, and compute the solution
$u^\epsilon_i = \mathcal{S}^\epsilon f_i \in\mathcal{Y}$ associated
with each of these configurations. These solutions form a point cloud
in a high dimensional space $\mathcal{Y}$. We subdivide the set of configurations $\{f_i\}$
into a number of small neighborhoods, and we look for the tangential
approximation to the mapping \eqref{eq:defS} on each of these
neighborhoods. Given a configuration $f$, we identify the neighborhood
to which it belongs, and interpolate linearly to obtain the
corresponding solution.

We summarize our online-offline strategy as follows. (Some
modifications described in Section~\ref{sec:algorithm} will reduce the
cost of implementation.)

\begin{itemize}
\item[] \textbf{Offline}: Randomly sample $f_i(x)$, $i = 1,\dotsc,N$,
  and find solutions $u^\epsilon_i = \mathcal{S}^\epsilon f_i$;
  \item[] \textbf{Online}: Given $f(x)$:
  \begin{itemize}
    \item[] \textbf{Step 1}: Identify the $k$-nearest neighbors of
      $f(x)$, call them $f_{i_j}, j = 1,2,\dotsc,k$, with $f_{i_1}$
      being the nearest neighbor;
    \item[] \textbf{Step 2}: Compute
    \[
    \mathcal{S}^\epsilon \phi \approx u^\epsilon_{i_1} +\mathsf{U}\cdot {c}\,,\quad \text{with} \quad\mathsf{U} =
    \begin{bmatrix}
    | & & | \\ u^\epsilon_{i_2}-u^\epsilon_{i_1} & \dots & u^\epsilon_{i_k}-u^\epsilon_{i_1} \\ | & & |
    \end{bmatrix},
    \]
    where $c$ is a set of coefficient that fits $f-f_{i_1}$ with a
    linear combination of $f_{i_j} - f_{i_1}$, for $j = 2,3,\dotsc,k$.
  \end{itemize}
\end{itemize}

In \textbf{Step 2} we used the fact that the solution manifold if of low dimensional locally. To make the strategy mathematically precise, we need to address
several questions, including the following.
\begin{itemize}
\item How should we sample $f_i(x)$ during the offline step?
\item What metric should we use to quantify distance?
\item Since computing each solution map $u^\epsilon_i = \mathcal{S}^\epsilon f_i$ is
  expensive, is there anyway to reduce the cost further?
\end{itemize}
We discuss these questions in the following sections. We stress that the manifold learning technique that we investigate in this paper works best when the intrinsic dimensionality of the problem is significantly smaller than the typical required degrees of freedom, and this holds true for all homogenizable problems where the discretization of the limiting equation eliminates the $\epsilon$ dependence. For problems without $\epsilon$ dependence, and the dimension of the numerical solution is only moderately large, the approach that we take is not expected to reduce cost.

\subsection{The layout of the paper}
We discuss the general recipe of the algorithm in
\Cref{sec:algorithm}, then show how the approach can be applied to two
examples (a semi-linear elliptic equation and a nonlinear radiative
transfer equation coupled with a temperature term) in
\Cref{sec:LowRank_elliptic} and \Cref{sec:LowRank_RTE},
respectively. In both sections, we review the relevant homogenization
theory for the equations, study the low rank structure of the
tangential solution spaces, and present numerical evidence for the
efficacy of our approach.

\section{Framework}\label{sec:algorithm}

Our approach is a domain decomposition algorithm that makes use of
Schwarz iteration.

After decomposing the domain into multiple overlapping patches, the
Schwarz method solves the PDE in each patch, conditioned on agreement
of solutions in the overlapping regions, which are boundary regions
for the adjacent patches. At the initial step, these boundary
conditions are unknown, so some initial guess is made. Subsequently,
solution of PDEs on each patch alternates with updates of the solution
on the overlapping regions, until convergence is obtained with respect
to certain criteria. The cost of the entire process is determined by
the number of iterations and the cost of the local solves, noting
that, as with any domain decomposition method, the local solves can be
performed in parallel. The approach is efficient when the local solves
can be performed much more efficiently on the available computing
resources than a solver that does not decompose the domain. The
optimal domain partitioning depends on the conditioning of the problem
and is often specific to the problem under study. Comprehensive
descriptions of the Schwarz method appear
in~\cite{SmBjGr:2004,ToWi:2006}.

This basic Schwarz iteration does not fully address the issue of
$\eps$-dependence that we discussed in \Cref{sec:intro}, since local solvers still necessarily depend on $\eps$. As a step
toward making use of compression, we take the viewpoint that the
purpose of the local solution step is to implement a
boundary-to-boundary map, taking one part of the boundary conditions
on a patch and using the solution of the resulting PDE to update the
boundary conditions for its neighboring patches.
We propose to learn the boundary-to-boundary maps in an ``offline''
stage, by running the local solvers as many times as are needed to
attain the desired accuracy in this map.
This offline stage comes with a high overhead cost, but the
computation is done only once, and we hope that the cost of the online
stage is greatly reduced by having the boundary-to-boundary maps
available. Note that this ``offline" learning process is distinct from
the \textbf{Offline} stage discussed in Section 1. With the
application of domain decomposition, it is the local behavior that
needs to be learned, instead of the full $u^\epsilon$.

In the linear setting, building the boundary-to-boundary maps is quite
straightforward. It amounts roughly to finding all discrete Green's
functions, with the degree of freedom being determined by the number
of grid points on the patch boundary, with one Green's function per
grid point. In the nonlinear setting, the boundary-to-boundary map is
nonlinear, so we can no longer build a linear basis, and we turn to
manifold learning approach to approximate the map. Specifically, in
the offline stage, we would sample randomly some configurations and
find the corresponding image under the map. The resulting point cloud
in high dimensional space can be viewed as samples of the manifold,
which we can then learn by means of local approximate tangential
planes. In the online stage, these tangential planes are used as
surrogates to local boundary-to-boundary maps.

Before presenting details of the offline and online stage
computations, we specify the setup and notation.  We consider the
following nonlinear PDE with Dirichlet boundary conditions in a domain
$\Omega\subset\mathbb{R}^2$:
\begin{equation}\label{eqn:general}
\begin{cases}
\mathcal{N}^\eps u^\eps = 0\,, & \quad \text{in }\Omega \\
u^\eps = \phi, & \quad \text{on } \partial\Omega,
\end{cases}
\end{equation}
where, as usual,  $\eps$ indicates the small scale of the problem.  For
simplicity, we will assume throughout a square geometry $\Omega = [0,
  L]^2$.  The domain $\Omega$ is decomposed into overlapping
rectangular patches defined by
\begin{equation}
\Omega = \bigcup_{m\in J} \Omega_m,\quad \text{with}\quad \Omega_m = (x_{m_1}^{(1)},x_{m_1}^{(2)})\times(y_{m_2}^{(1)},y_{m_2}^{(2)})\,,
\end{equation}
where $m = (m_1,m_2)$ is a multi-index and $J$ is the collection of the indices
\[
J = \{m = (m_1,m_2):\,m_1 = 1,\dots,M_1,\,m_2 = 1,\dots,M_2\}\,.
\]
This setup is illustrated in Figure~\ref{fig:elliptic_decomp}. For
each patch we define the associated partition-of-unity function
$\chi_m$, which has $\chi_m(x) \geq 0$ and
\begin{equation}\label{eqn:POU_elliptic}
\chi_m(x) = 0\quad \text{on}\,\, x \in \Omega\backslash\Omega_m\,,\quad \sum_m\chi_m(x) = 1, \quad \forall x \in \Omega\,.
\end{equation}
We set $\partial\Omega_m$ to be the boundary of patch $\Omega_m$ and
denote by $\mathscr{N}(m)$ the collection of indices of the neighbors
of $\Omega_m$. In this particular 2D case, we have
\begin{equation}
\mathscr{N}(m) = \{(m_1\pm 1, m_2)\} \cup \{(m_1,m_2\pm 1)\} \subset J\,.
\end{equation}

\begin{figure}[htbp]
  \centering
  \includegraphics[width=0.45\textwidth]{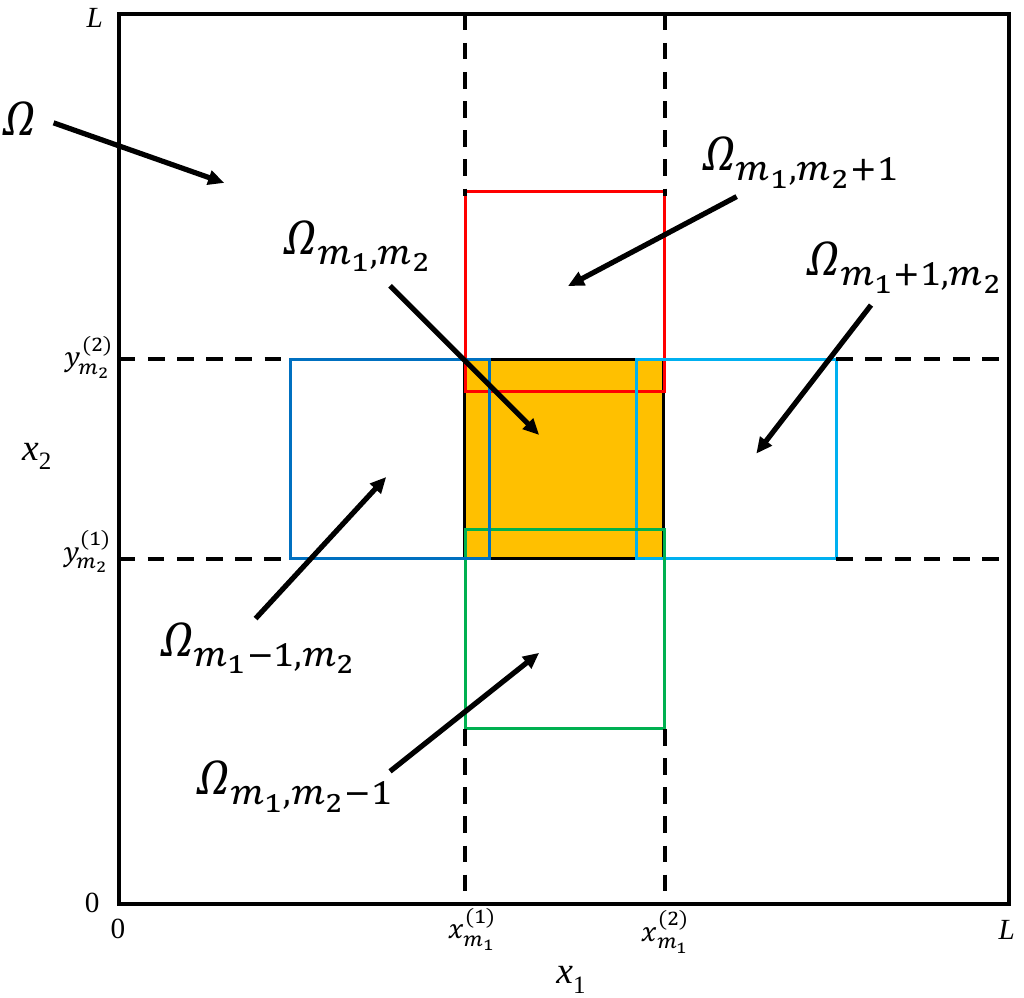}
  \includegraphics[width=0.45\textwidth]{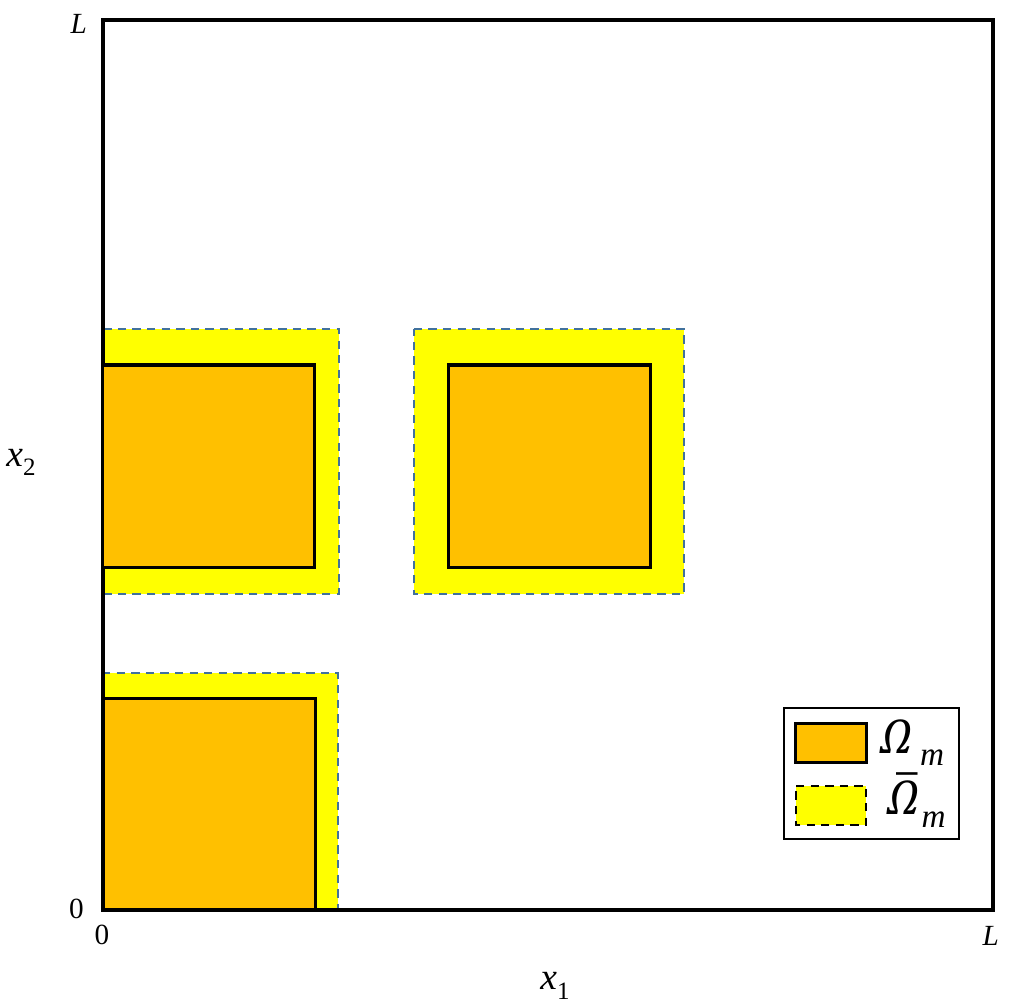}
  \caption{The plot on the left shows the domain decomposition for a square geometry. Each patch is labeled by a multi-index $m =
    (m_1,m_2)$. The adjacent patches of $\Omega_m$ are defined to be
    the patches on its north/south/west/east sides. The plot on the
    right demonstrates the use of local enlargement to damp boundary
    effects.}
  \label{fig:elliptic_decomp}
\end{figure}

Assume that the equation~\cref{eqn:general} is well-posed, meaning
that given $\phi$ in some function space $\mathcal{X}$, there exists a
unique solution $u^\eps$ in another function space
$\mathcal{Y}$. Assume further that the local nonlinear equation on
patch $\Omega_m$ defined by
\[
  \begin{cases}
  \mathcal{N}^\eps u_m^\eps = 0\,, & \; \text{in }\Omega_m\,, \\
  u_m^\eps = \phi_m, & \; \text{on } \partial\Omega_m\,,
  \end{cases}
\]
is well-posed, given local boundary condition $\phi_m$ in some
function space $\mathcal{X}_m$, and that the solution $u_m^\eps$ lives
in space $\mathcal{Y}_m$. We further define the following operators.
\begin{itemize}
  \item $\mathcal{S}^\epsilon_m$ denotes the solution operator that maps local
    boundary condition $\phi_m$ to the local solution $u_m^\eps$:
  \[
  \mathcal{S}^\epsilon_m: \mathcal{X}_m \to \mathcal{Y}_m, \quad \mathcal{S}^\epsilon_m \phi_m = u^\epsilon_m\,.
  \]

  \item $\mathcal{I}_{m}^{l}$ denotes the trace
    operator for all $l\in\mathscr{N}(m)$:
  \[
  \mathcal{I}_{m}^{l}u_l^\eps = u_l^\eps|_{\partial\Omega_m\cap\Omega_l}, \quad l\in\mathscr{N}(m)\,,
  \]
  which takes the value of $u^\eps_l$ restricted on the boundary
  $\partial\Omega_m\cap\Omega_l$. Here we assume that the space
  $\mathcal{Y}_l$ allows for trace.

  \item $\mathcal{P}_m$ denotes the boundary update operator, mapping
    $\bigoplus_{l\in\mathcal{N}(m)}\mathcal{X}_{l}$ to $\mathcal{X}_m$
  \[
    \mathcal{P}_m(\phi_{l},\, l\in\mathscr{N}(m)) =
    \begin{cases}
      \mathcal{I}_m^{l}\mathcal{S}^\epsilon_{l}\phi_{l}, & \; \text{on } \partial\Omega_m\cap\Omega_{l}, \quad l\in\mathscr{N}(m)\,,\\
      \phi|_{\partial\Omega_m\cap\partial\Omega}, &
      \; \text{on } \partial\Omega_m\cap\partial\Omega\,.
    \end{cases}
  \]
  Note that on the points in $\partial \Omega_m \cap \partial \Omega$,
  the boundary condition from the whole domain $\Omega$ is imposed.
\end{itemize}

The offline and online stage of the algorithm are essentially to
construct and to evaluate $\mathcal{P}_m$, as we now show.

\subsection{Offline Stage}\label{sec:framework_offline}

The goal of the offline stage is to construct a dictionary to
approximate $\mathcal{P}_m$ for every $m\in J$. To eliminate any
boundary layer effect, we enlarge each local patch slighly by adding a
margin around its edges (except for the edges that correspond to part
of the boundary of the whole domain). The enlarged domains are denoted
by $\wt{\Omega}_m$ and illustrated in
Figure~\ref{fig:elliptic_decomp}.

We denote by $\wt{\mathcal{X}}_m$ the space of boundary conditions on
$\partial\wt{\Omega}_m$ equipped with norm $\|\cdot\|$, and define a
ball in $\wt{\mathcal{X}}_m$ as follows:
\[
B(R_m;\,\wt{\mathcal{X}}_m) = \{\wt{\phi}\in\wt{\mathcal{X}}_m: \|\wt{\phi}\|\leq R_m\}\,.
\]
First, we draw $N$ samples randomly from the ball, as  follows:
\[
\wt{\phi}_{m,i}\in B(R_m;\,\wt{\mathcal{X}}_m)\,,\quad i = 1,\dots,N\,.
\]
(The specific measure used in drawing depends on the particular
problem being considered; we will make it more precise in the examples
below.) For these samples we obtain local solutions $\wt{u}_{m,i}$
from the following PDEs:
\begin{equation}\label{eqn:general_local}
\begin{cases}
\mathcal{N}^\eps \wt{u}_{m,i}^\eps = 0\,, & \quad \text{in }\wt{\Omega}_m \\
\wt{u}_{m,i}^\eps = \wt{\phi}_{m,i}\,,  & \quad \text{on } \partial\wt{\Omega}_m
\end{cases}
\end{equation}
We build a dictionary from these solutions by confining them in the
interior $\Omega_m$ and the boundary $\partial\Omega_{m}$:
\begin{equation}\label{eqn:general_dic}
\mathscr{I}_m = \{\psi_{m,i}=\wt{u}_{m,i}|_{\Omega_m}\}_{i=1}^N,\quad \mathscr{B}_m = \{\phi_{m,i}=\wt{u}_{m,i}|_{\partial\Omega_m}\}_{i=1}^N\,.
\end{equation}
Since the problems that we consider are homogenizable, meaning that
the solution manifold is of low dimensional, the value of $N$ can be
relatively small.

\begin{remark}
Two remarks are in order.
\begin{itemize}
\item How to sample? That is, how to find a measure $\mu_m$ on
  $\wt{\mathcal{X}}_m$ for drawing samples? To make the setting more
  precise, we discretize the space $\wt{\mathcal{X}}_m$ to get
  $\wt{\mathcal{X}}^h_m$ equipped with norm $\|\cdot\|_h$, and define
  a measure $\mu_m^h$ on the ball $B(R_m;\,\wt{\mathcal{X}}^h_m)$.
  Denoting the dimension of $\wt{\mathcal{X}}_m$ by $p$, we sample the
  magnitude and the angle separately, that is, we take the measure as
  a product $\mu_m^h = \mu_{r,m}\otimes\mu_{S,m}$ with $\mu_{r,m}$
  being the radial part on $(0,R_m)$ and $\mu_{S,m}$ being the measure
  on the unit sphere $S^{p-1} = \{\phi\in\wt{\mathcal{X}}^h_m:
  \|\phi\|_h=1\}\subset\Rbb^p$. The angular measure $\mu_{S,m}$ is
  chosen to be the uniform and the radial part $\mu_{r,m}$ has a
  density function $f(r) = \frac{D+1}{R_m^D}r^D$. The number $D$ here
  plays the role of effective dimension; it should depend on the
  expected dimension of solution manifold. Note that if we take $D =
  p-1$, the measure $\mu_m^h$ is exactly the uniform measure on the
  full ball $B(R_m;\,\wt{\mathcal{X}}^h_m)\subset\Rbb^p$. The question
  of selecting $D$ in a rigorous way is left to future
  research. (See~\Cref{app:elliptic} and~\Cref{app:nrte} for further
  details on this issue.)

  \item The physical boundary. To respect the boundary condition on
    $\partial\Omega$, the boundary patches $\Omega_m$ that touch the
    physical boundary need to be treated differently. For each sample
    $\wt{\phi}_{m,i}$, the physical boundary condition is enforced on
    the set
    $\partial\Omega_m\cap\partial\Omega$. Random sampling is done only
    on the remaining part of the patch boundary, that is,
    $\partial\Omega_m\backslash\partial\Omega$. See~\Cref{app:elliptic}
    and~\Cref{app:nrte} for details.
\end{itemize}
\end{remark}

\subsection{Online Stage}
The online stage finds a particular solution $u$ for given boundary
data $\phi$, based on information accumulated in the offline
stage. This process is carried out through a Schwarz iteration to
update local boundary conditions on each patch.

Denote by $\phi^{(n)}=[\cdots\,,\phi_{m}^{(n)}\,,\cdots]$ the
collection of local boundary conditions at the $n$th iteration, with
$m$ being the patch index. At each iteration, we need to obtain
$\phi_{m}^{(n+1)}=\mathcal{P}_m\phi^{(n)}$. For each $m\in J$, let
$\phi_{m,i_q^{(n)}}$ be the $q$-th $L^2$-nearest neighbor of
$\phi_m^{(n)}$ in $\mathscr{B}_m$, $q = 1,2,\dotsc,k$. These
neighbors, supported on $\partial\Omega_m$ lie (approximately) on a
local tangential plane centered at $\phi_{m,i_1^{(n)}}$:
\begin{equation}\label{eqn:local_basis_bdy_general}
\Phi_{m}^{(n)} =
\begin{bmatrix}
| & & |\\
\phi_{m,i_2^{(n)}}-\phi_{m,i_1^{(n)}} & \dots & \phi_{m,i_k^{(n)}}-\phi_{m,i_1^{(n)}} \\
| & & |
\end{bmatrix}\,.
\end{equation}
Associated with this plane, we also formulate the solution space
centered around $\psi_{m,i_1^{(n)}}$:
\begin{equation}\label{eqn:tan_space_general}
\Psi_{m}^{(n)} =
\begin{bmatrix}
| & & | \\
\psi_{m,i_2^{(n)}}-\psi_{m,i_1^{(n)}} & \dots & \psi_{m,i_k^{(n)}}-\psi_{m,i_1^{(n)}} \\
| &  & |
\end{bmatrix}\,.
\end{equation}
Locally, the map between these two planes is approximately linear, and
thus to find $\phi_{m}^{(n+1)}=\mathcal{P}_m\phi^{(n)}$, we look for a
linear interpolation of $\phi_m^{(n)}$ on $\Phi_m^{(n)}$, and map this
interpolation to $\Psi^{(n)}_m$. More precisely, we look for
$c_m^{(n)}$ that solves the least-squares problem
\begin{equation}\label{eqn:coef_general}
c_m^{(n)} = \argmin_{v_m\in\mathbb{R}^{k-1}}\big\|\phi_m^{(n)}-\phi_{m,i_1^{(n)}}-\Phi_{m}^{(n)}v_m \big\|_{L^2(\partial\Omega_m)}\,,
\end{equation}
and define the approximate solution to be:
\begin{equation}\label{eqn:general_approx_soln_op}
  u_m^{(n)} ={\mathcal{S}^\epsilon}_m\phi_m^{(n)} \approx \psi_{m,i_1^{(n)}}+\Psi_{m}^{(n)}c_m^{(n)}\,.
\end{equation}

To summarize: the map $\mathcal{P}_{m}\phi^{(n)}$ is a composition of
${\mathcal{P}}_{m}(\phi_l^{(n)},\,l\in\mathscr{N}(m))$ with
$l\in\mathscr{N}(m)$, where
\begin{equation}\label{eqn:patch2_general}
\phi_{m}^{(n+1)} = {\mathcal{P}}_{m}(\phi_l^{(n)},\,l\in\mathscr{N}(m)) = \mathcal{I}_{m}^{l}{\mathcal{S}^\epsilon}_l\phi_l,\quad\text{on } \partial\Omega_m\cap\Omega_{l}\,.
\end{equation}

Once a preset error tolerance is achieved, at some step $n$ (usually
because the local boundary condition barely changes), the global
solution is patched up as follows:
\begin{equation}\label{eqn:general_global_soln}
u^{(n)} = \sum_{m\in J} \chi_m u_m^{(n)}\,,
\end{equation}
where $u_m^{(n)}$ is the local
solution~\cref{eqn:general_approx_soln_op} and
$\chi_m:\Omega\rightarrow\mathbb{R}$ is the smooth partition of unity
associated with the partition.

We summarize the procedure in Algorithm~\ref{alg:general}.

\begin{remark}
  The Johnson-Lindenstrauss lemma~\cite{JoLi:1984} indicates that the
  search for $d$-dimensional $k$ nearest neighbors in a data set of
  size $N$, with distance error $\delta$, can be done in query time
  $O\Bigl(kd\frac{\log N}{\delta^2}\Bigr)$ and storage cost
  $N^{O(\log(1/\delta)/\delta^2)}+O\Bigl(d\bigl(N+\frac{\log
    N}{\delta^2}\bigr)\Bigr)$~\cite{InMo:1998,AnInRa:2018}. In addition,
    a cost of $O(k^2 d)$ is incurred at each iteration, due to $L^2$
    minimization for each patch via QR factorization. In our
  setting, $d$ is equal to the degrees of
  freedom on the boundary $\partial\Omega_m$.
\end{remark}

\begin{remark}
To avoid notational complexity, the discussion above does not consider
the physical boundary $\partial \Omega$. If a patch contains part of
$\partial \Omega$, then that particular section of the patch is not
updated. The true boundary condition $\phi$ is enforced in every
iteration. The derivation is straightforward and is omitted from the
discussion.
\end{remark}

\begin{algorithm}
\caption{Multiscale solver for nonlinear homogenizable
  equations~\cref{eqn:general}.}\label{alg:general}

\begin{algorithmic}[1]
\State Given the radius $R_m$, the number of nearest neighbors $k$, the tolerance $\delta$ and the initial guess of boundary conditions $\phi_{m}^{(0)}$ on each patch $m \in J$.
\State {\bf Domain Decomposition:}
\State \hspace*{0.2in} Decompose $\Omega$ into overlapping patches: $\Omega=\bigcup_{m\in J}\Omega_m$, and enlarge each patch to obtain $\wt{\Omega}_m$.

\State {\bf Offline Stage:} Prepare local dictionaries on interior patches $\Omega_m$.
\Indent
   \State Step 1: For each $m\in J_i$, generate $N$ samples $\wt{\phi}_{m,i}$ from $B(R_m;\wt{\mathcal{X}}_m)$;
   \State Step 2: For all $i$, call function
   \[
   \wt{u}_{m,i}=\text{LocPDESol}(\wt{\Omega}_m,\wt{\phi}_{m,i})\,;
   \]
   \State Step 3: Collect local dictionaries according to~\cref{eqn:general_dic} for $\mathscr{B}_m$ and $\mathscr{I}_m$.
\EndIndent

\State {\bf Online Stage:} Schwarz iteration.
\Indent
\While{$\sum_{m}\|\phi_{m}^{(n)}-\phi_{m}^{(n-1)}\|_{L^2(\partial\Omega_m)}\geq\delta$}
\For{$m \in J$} \State Search for $k$-nearest neighbors of
$\phi_{m}^{(n)}$ in $\mathscr{B}_m$; \State Solve $c_m^{(n)}$ from the
least-squares problem~\cref{eqn:coef_general}; \State Update
$\phi_{m}^{(n+1)}$ by~\cref{eqn:patch2_general}.  \EndFor \State
$n \gets n+1$ \EndWhile \EndIndent \\\Return Global solution $u^{(n)}$
defined by~\cref{eqn:general_global_soln}.

\algrule

\setcounter{ALG@line}{0}
\Function{LocPDESol}{Local domain $\Omega_m$, Boundary condition $\phi_m$}
\State Perform the standard finite difference or finite element methods to solve the local nonlinear equation~\cref{eqn:general_local};
\\\Return Local solution $u_m$
\EndFunction

\end{algorithmic}
\end{algorithm}

\section{Example 1: Semilinear elliptic equations with highly oscillatory media}\label{sec:LowRank_elliptic}
In this section, we apply the methodology described above to solve
semilinear elliptic equations. Semilinear elliptic equations with
multiscale structures arise in a variety of situations, for instance,
in nonlinear diffusion generated by nonlinear
sources~\cite{JoLu:1973}, and in the gravitational equilibrium of
stars~\cite{Li:1982,Ch:1957}.  As fundamental models in many areas of
physics and engineering, the equations have received considerable
attention.

We consider the equation
\begin{equation}\label{eqn:semi_elliptic}
\begin{cases}
-\nabla_x\cdot\left(A\left(x,\tfrac{x}\eps\right)\nabla_x u^\eps\right) + f(u^\eps) = 0,\quad& x\in\Omega\,, \\
u^\eps(x) = \phi(x),\quad& x\in\partial\Omega\,.
\end{cases}
\end{equation}
The physical domain is $\Omega\subset \mathbb{R}^d$ with $d\geq1$, and
Dirichlet boundary condition is given as $\phi(x)$. The permeability
$A(x,y) = (a_{ij}(x,y))_{d\times
  d}:\Omega\times\mathbb{R}^d\to\mathbb{R}^{d\times d}$ depends on
both the slow variable $x$ and the fast variable $y=x/\eps$ and is
highly oscillatory. The function $f:\mathbb{R}\rightarrow\mathbb{R}$
describes the nonlinear source term. The solution
$u^\eps$ presents one component in a chemical reaction or one species
of a biological system.

The well-posedness of equation~\cref{eqn:semi_elliptic} is
classical. We assume that the permeability $A$ is a symmetric matrix
with $L^{\infty}$-coefficients satisfying the standard coercivity
condition, and that the nonlinear function $f$ is locally Lipschitz
continuous and increasing. Then, assuming the boundary
$\partial\Omega$ is smooth enough, given boundary condition $\phi\in
H^{1/2}(\partial\Omega) \cap L^{\infty}(\partial\Omega)$, the
problem~\cref{eqn:semi_elliptic} has a unique $H^1$-solution
satisfying the maximum principle. We refer to~\cite{GiTr:2015,Ch:2009}
for details.

\subsection{Homogenization limit}
The semilinear elliptic equation~\cref{eqn:semi_elliptic} has a
homogenization limit as $\eps\rightarrow0$. Supposing that $A(x,y)$ is
smooth and periodic in $y$ with period $I = [0,1]^d$, then as
$\eps\rightarrow0$, the solution $u^\eps$ converges to a limit
$u^\ast$ that satisfies the same class of semilinear elliptic
equations with an $\eps$-independent effective permeability $A^\ast(x)
= (a_{ij}^\ast(x))_{d\times d}$:
\begin{equation}\label{eqn:homo_elliptic}
\begin{cases}
-\nabla_x\cdot(A^\ast(x)\nabla_x u^\ast) + f(u^\ast) = 0, \quad &x\in \Omega\,,\\
u^\ast(x) = \phi(x), \quad &x\in \partial\Omega\,.
\end{cases}
\end{equation}
This equation (in particular, the effective permeability $A^\ast(x)$)
can be derived by expanding the equation~\cref{eqn:semi_elliptic} into
different orders of $\eps$. Rigorous proofs are given in
\cite{BeLiPa:2011,BeBoMu:1992,Pa:1999}. We cite the following theorem as a
reference:
\begin{theorem}[Section~16.3 in Chapter~1 of~\cite{BeLiPa:2011}; see
    also~\cite{Al:1992}]\label{thm:bensoussan} Assume the boundary
  $\partial\Omega$ is smooth.  Given $\phi(x)\in
  H^{1/2}(\partial\Omega)\cap L^{\infty}(\partial\Omega)$, let
  $u^\eps$ be the unique solution to the semilinear elliptic
  equation~\cref{eqn:semi_elliptic} in $H^1(\Omega)\cap
  L^{\infty}(\Omega)$.  Assume that the permeability $A(x,y)$ is
  periodic in $y$ with period $I = [0,1]^d$ and that $A(x,\cdot)\in
  C^1(I)$.  Then the solution $u^\eps$ converges weakly in
  $H^1(\Omega)$ as $\eps\rightarrow0$ to $u^\ast$ (the solution
  to~\cref{eqn:homo_elliptic}), where the permeability $A^\ast(x) =
  (a_{ij}^\ast(x))_{d\times d}$ is defined by
\begin{equation}\label{eqn:a_ast}
a_{ij}^\ast(x) = \int_I \sum_{k,l} a_{kl}(x,y)(\delta_{ki}+\partial_{y_k}\chi_i)(\delta_{lj}+\partial_{y_l}\chi_j)\mathrm{d}y\,.
\end{equation}
Here, for each fixed coordinate $j=1,2,\dotsc,d$, the function
$\chi_j(x,y)$ is the solution of the following cell problem with
periodic boundary condition on $I$:
\begin{equation}\label{eqn:chi}
\nabla_y\cdot(A(x,y)\nabla_y(\chi_j(x, y)+y_j)) = 0\,.
\end{equation}
\end{theorem}

To solve~\cref{eqn:semi_elliptic}, the discretization has to resolve
$\eps$, but in the limit~\cref{eqn:homo_elliptic}, the discretization
is independent of $\eps$. This suggests significant opportunities for
cost savings: The information contained in $O(1/\eps)$ degrees of
freedom can be expressed with $O(1)$ degrees of freedom.

The literature for numerical homogenization is rich, particularly for
the linear setting when $f=0$. Relevant approaches include the
multiscale finite element method
(MsFEM)~\cite{HoWu:1997,EfHoWu:2000,HoWuCa:1999}, the heterogeneous
multiscale method (HMM)~\cite{EWEn:2003,AbSc:2005,EMiZh:2005}, the
generalized finite element method~\cite{BaMe:1997,BaLi:2011},
upscaling based on harmonic coordinates~\cite{OwZh:2007}, elliptic
solvers based on $\mathcal{H}$-matrices~\cite{Be:2007,Ha:2015}, the
reduced basis method~\cite{AbBaVi:2015,AbBa:2012}, the use of
localization~\cite{MaPe:2014}, and the methods based on random
SVD~\cite{ChLiLuWr:2018randomized,ChLiLuWr:2018,ChLiLuWr:2019}, to
name a few. The analytical understanding of the homogenized equation
is essential in the construction of these
methods~\cite{Al:1992}. When randomness presents, one can also look for low dimensional representation of the solutions in the random space~\cite{HOU2017375,LiZhangZhao_dimension_reduction,HouLiZhang_low_dimension,ChEfLeZh:2018cluster}.

The literature for nonlinear problems is not as rich. There are
several works on quasilinear problems, all of which can be seen as
extensions of classical methods, including the
MsFEM~\cite{EfHoGi:2004,ChSa:2008,EfHo:2009}, the
HMM~\cite{EMiZh:2005,AbVi:2014}, the generalized finite element
method~\cite{EfGaLiPr:2014}, the local orthogonal decomposition
method~\cite{HeMaPe:2014}, the reduced basis
method~\cite{AbBaVi:2015} and nonlocal multicontinua upscaling~\cite{ChEfLeWh:2018nonlinear}. These solvers must be designed carefully
for specific nonlinear equations. By contrast, our method makes use of
the low-rankness of the solution sets and could be applied with minor
modification to different equations.

\subsection{Low dimensionality of the tangent space}
We now study the structure of the tangent space of the solution manifold,
verifying in particular the low dimension assumption. We choose some
point $\ol{u}^\eps$ on the solution manifold and then randomly pick a
neighboring solution point $u^\eps$. These two points are solutions
to~\cref{eqn:semi_elliptic} computed from distinct nearby boundary
configurations $\ol{\phi}$ and $\phi$, that is,
\begin{equation}\label{eqn:point_cloud_bdy_elliptic}
\ol{u}^\eps|_{\partial\Omega} = \ol{\phi}\,,\quad u^\eps|_{\partial\Omega} = \phi\,,\quad\text{with}\quad \|\ol{\phi}-\phi\|_{L^{\infty}(\partial\Omega)} = \mathcal{O}(\delta)\,.
\end{equation}
By varying $\phi$ around $\ol{\phi}$, one can build a small point
cloud around $\ol{u}^\eps$.  Denoting $\delta u^\eps := u^\eps - \ol{u}^\eps$,
we have immediately that
\begin{equation}\label{eqn:delta_elliptic}
\begin{cases}
-\nabla_x\cdot\left(A\left(x,\tfrac{x}\eps\right)\nabla_x \delta u^\eps\right) + f(\ol{u}^\eps+\delta u^\eps) - f(\ol{u}^\eps) = 0, \quad &x\in \Omega\,,\\
\delta u^\eps(x) = \phi(x) - \ol{\phi}(x), \quad &x\in \partial\Omega\,.
\end{cases}
\end{equation}
In the small-$\delta$ regime, this collection of solution differences
$\delta u^\eps$ spans the tangent plane. We claim this tangent plane
is low dimensional, so that it inherits the homogenization effect of
the original equation. We have the following result.
\begin{theorem}\label{thm:elliptic_homogenization_point_cloud}
Let $\delta u^\eps$ solve~\cref{eqn:delta_elliptic}. Assume $A(x,y) =
(a_{ij}(x,y))_{d\times d}$ is periodic in $y$ with period $I =
[0,1]^d$. The equation has homogenization limit when $\eps\to0$,
meaning there exists a limiting permeability $A^\ast(x) =
(a_{ij}^\ast(x))_{d\times d}$, determined by $A(x,y)$ via
equation~\cref{eqn:a_ast} and~\cref{eqn:chi}, so that
$\delta u^\eps\to \delta u^\ast$ and $\delta u^\ast$ solves:
\begin{equation}\label{eqn:asymp_elliptic}
\begin{cases}
-\nabla_x\cdot(A^\ast(x)\nabla_x \delta u^\ast) + f(\ol{u}^\ast+\delta u^\ast) - f(\ol{u}^\ast) = 0, \quad &x\in \Omega\,,\\
\delta u^\ast(x)  = \phi(x) - \ol{\phi}(x), \quad &x\in \partial\Omega\,,
\end{cases}
\end{equation}
where $\ol{u}^\ast$ solves:
\begin{equation}\label{eqn:asymp_center_elliptic}
\begin{cases}
-\nabla_x\cdot(A^\ast(x)\nabla_x \ol{u}^\ast) + f(\ol{u}^\ast) = 0, \quad &x\in \Omega\,,\\
u^\ast(x) = \ol{\phi}(x), \quad &x\in \partial\Omega\,.
\end{cases}
\end{equation}
Further, for small $\delta$, equation~\cref{eqn:asymp_elliptic}, in the leading order of $\delta$, becomes:
\begin{equation}\label{eqn:linearized_elliptic}
-\nabla_x\cdot(A^\ast(x)\nabla_x \delta u^\ast) + f'(\ol{u}^\ast(x))\delta u^\ast = 0\,.
\end{equation}
\end{theorem}
\begin{proof}
By applying~\Cref{thm:bensoussan} to the equation for $\ol{u}^\eps$, which is
\[
\begin{cases}
-\nabla_x\cdot\left(A\left(x,\tfrac{x}\eps\right)\nabla_x \delta \ol{u}^\eps\right) + f(\ol{u}^\eps) = 0, \quad &x\in \Omega\,,\\
\delta \ol{u}^\eps(x) = \ol{\phi}(x), \quad &x\in \partial\Omega\,,
\end{cases}
\]
we have by comparing with equation~\cref{eqn:semi_elliptic} for
$u^\eps$ that $\ol{u}^\eps$ converges weakly to $\ol{u}^\ast$, which
solves~\cref{eqn:asymp_center_elliptic}, and that $u^\eps$ converges
weakly to $u^\ast$, which solves~\cref{eqn:homo_elliptic}. From the
definition $\delta u^{\eps} = u^\eps- \ol{u}^\eps$, we find that
$\delta u^\eps$ converges to $\delta u^\ast$, which
solves~\cref{eqn:asymp_elliptic}.
\end{proof}

This theorem suggests that for the discretized equation, because of
the existence of the homogenized limit, the tangent plane of the
discrete solution is approximately low-rank.  The space spanned by
$\{\delta u^\eps\}$ can be approximately spanned by $\{\delta
u^\ast\}$, which solves the limiting
equation~\cref{eqn:asymp_elliptic} without dependence on small scales.

\subsection{Implementation}
We apply Algorithm~\ref{alg:general} to equation
\eqref{eqn:semi_elliptic} with $f(u) = u^3$ and
$\Omega=[0,L]^2\subset\mathbb{R}^2$, that is,
\begin{equation}\label{eqn:elliptic_2d}
\begin{cases}
-\nabla_x\cdot\left(a\left(x,\tfrac{x}\eps\right)\nabla_x u\right) + u^3 = 0\,, \quad &x\in\Omega= [0,L]^2, \\
u(x) = \phi(x), \quad &x\in \partial\Omega.
\end{cases}
\end{equation}
We use the domain decomposition strategy of
Section~\ref{sec:algorithm} to solve this system. Since $\Omega$ is
convex and the coefficient $a(x,x/\eps)$ belongs to
$L^{\infty}(\Omega)$, we can show using the monotone
method~\cite{AmMo:1971,Ch:2009} that the equation is well-posed,
having a unique solution if we set
\[
\mathcal{X} = H^{1/2}(\partial\Omega)\cap L^{\infty}(\partial\Omega), \quad  \mathcal{Y} = H^1(\Omega)\cap L^{\infty}(\Omega)\,.
\]

In the offline stage, we generate $N$ samples for each enlarged patch
$\wt{\Omega}_m$, as follows:
\[
\wt{\phi}_{m,i}\in B(R_m;\,\wt{\mathcal{X}}_m)\,,\quad i = 1,\dots,N\,.
\]
(The measure we use for sampling is discussed in~\Cref{app:elliptic}.)
We equip the ball with $H^{1/2}$-norm:
\begin{equation}
B(R_m;\,\wt{\mathcal{X}}_m) = \{\wt{\phi}\in\wt{\mathcal{X}}_m: \|\wt{\phi}\|_{H^{1/2}(\partial\Omega_m)}\leq R_m\}\,.
\end{equation}
We compute the $H^{1/2}(\partial\Omega)$-norm numerically using the
Gagliardo seminorm~\cite{DiPaVa:2012}:
\[
\|\phi\|_{H^{1/2}(\partial\Omega)} = \sqrt{\int_{\partial\Omega}|\phi(x)|^2\mathrm{d}x + \iint_{\partial\Omega\times\partial\Omega}\frac{|\phi(x)-\phi(y)|^2}{|x-y|^2}\mathrm{d}x\mathrm{d}y}\,.
\]
For these boundary configurations, we solve the equation
\begin{equation}\label{eqn:patch_learn_elliptic}
\begin{cases}
-\nabla_x\cdot\left(a\left(x,\tfrac{x}\eps\right)\nabla_x \wt{u}_{m,i}\right) + \wt{u}_{m,i}^3 = 0\,,&\quad x\in \wt{\Omega}_m\,, \\
\wt{u}_{m,i}(x) = \wt{\phi}_{m,i}(x)\,,&\quad x\in\partial\wt{\Omega}_m\,,
\end{cases}
\end{equation}
and build two sets of dictionaries by confining the solutions in the
interior $\Omega_m$ and the boundary $\partial\Omega_{m}$, as follows:
\begin{equation} \label{eq:dics}
\mathscr{I}_m =
\{\psi_{m,i}=\wt{u}_{m,i}|_{\Omega_m}\}_{i=1}^N,\quad \mathscr{B}_m
= \{\phi_{m,i}=\wt{u}_{m,i}|_{\partial\Omega_m}\}_{i=1}^N\,.
\end{equation}

In the online stage, local boundary conditions are updated according
to~\cref{eqn:general_approx_soln_op} at each iteration, with
coefficients computed from~\cref{eqn:coef_general}. The local tangent
space is found by searching for the $k$ nearest neighbors in the
dictionary $\mathscr{B}_m$, mapped to the dictionary $\mathscr{I}_m$
(see \eqref{eq:dics}). We use the $L^2$ norm to measure the distance between the newly generated solutions and the older solution set.

Once a preset error tolerance is achieved (at step $n$, say), the
global solution is patched up from the local pieces, as follows:
\begin{equation}\label{eqn:elliptic_global_soln}
u^{(n)} = \sum_{m\in J} \chi_m u_m^{(n)}\,,
\end{equation}
where $u_m^{(n)}$ is the local solution on $\Omega_m$ at the $n$-th
step and $\chi_m:\Omega\rightarrow\mathbb{R}$ is a smooth partition of
unity.

\subsection{Numerical Tests}~\label{sec:num_elliptic}
We present numerical results for \eqref{eqn:elliptic_2d} in this
subsection. We use $L=1$, yielding the domain $\Omega = [0,1]^2$, and
define the oscillatory media as follows:
\[
a(x,y,x/\eps,y/\eps) = 2+\sin(2\pi x)\cos(2\pi y)+\frac{2+1.8\sin(2\pi
  x/\eps)}{2+1.8\cos(2\pi y / \eps)} + \frac{2+\sin(2\pi y
  /\eps)}{2+1.8\cos(2\pi x/\eps)}\,.
\]
The  boundary condition is
\begin{alignat*}{2}
\phi(x,0) &= -\sin(2\pi x)\,,\quad & \phi(x,1) & = \sin(2\pi x)\,,\\
\phi(0,y) &= \sin(2\pi y)\,,\quad & \phi(1,y) & = -\sin(2\pi y)\,.
\end{alignat*}

To form the partitioning, the whole domain $\Omega$ is divided equally
into $4\times4$ non-overlapping squares, and then each square is
enlarged by $\Delta x_{\text{o}} = .0625$ on the sides that do not intersect
with $\partial\Omega$, to create overlap. We thus have $M_1 = M_2 =
4$, with $\Omega_m$ for $m=(m_1,m_2)$, $m_1=1,2,3,4$ and
$m_2=1,2,3,4$, defined by
\begin{align*}
\Omega_m & =\left[\max\left(\tfrac{m_1-1}{M_1}-\Delta x_{\text{o}},0\right),\min\left(\tfrac{m_1}{M_1}+\Delta x_{\text{o}},1\right)\right] \\
& \quad \times \left[\max\left(\tfrac{m_2-1}{M_2}-\Delta x_{\text{o}},0\right),\min\left(\tfrac{m_2}{M_2}+\Delta x_{\text{o}},1\right)\right]\,,\quad m = (m_1,m_2) \in J\,.
\end{align*}

Denote $\Omega_m = [x_m^{(1)},x_m^{(2)}]\times[y_m^{(1)},y_m^{(2)}]$. The partition of unity function $\chi_m$ is defined by normalizing the bump functions on the overlapping domains. More precisely, we first define a bump function $f_m:\Omega\to\mathbb{R}$ supported on $\Omega_m$ as follows:
\[
f_m(x,y) =
\begin{cases}
\exp\left( -\frac{1}{1-|x-x_m|/\alpha_m} -\frac{1}{1-|y-y_m|/\beta_m} \right) \,, \quad (x,y)\in\Omega_m\\
0 \,, \quad \text{Otherwise}
\end{cases}\,,
\]
where $x_m = \frac{x_m^{(2)}-x_m^{(1)}}{2}$, $y_m = \frac{y_m^{(2)}-y_m^{(1)}}{2}$, $\alpha_m = \frac{x_m^{(1)}+x_m^{(2)}}{2}$ and $\beta_m = \frac{y_m^{(1)}+y_m^{(2)}}{2}$. The partition of unity $\chi_m:\Omega\to\mathbb{R}$ is then obtained by
\[
\chi_m(x,y) = \frac{f_m(x,y)}{\sum_{m\in J} f_m(x,y)} \,.
\]

A standard finite-volume scheme with uniform grid is used for
discretization, the corresponding nonlinear discrete system being
solved by Newton's method. The reference solutions are computed on the
fine mesh with $h = 2^{-12} = \tfrac{1}{4096}$. Unless otherwise
specified, other computations are performed with mesh size $h =
2^{-9} = \tfrac{1}{512}$.  Denoting the numerical solution by $u_{ij}
\approx u(x_i,y_j)$, we use the classical discrete $L^2$ norm
\[
\|u\|_{L^2} =  h \sqrt{ \sum_{i,j=0}^{p} |u_{ij}|^2 }\,,
\]
and the energy norm
\[
\|u\|_{\mathcal{E}} = h \sqrt{ \sum_{j=0}^{p}\sum_{i=0}^{p-1} a_{i+1/2,j} \left| \frac{u_{i+1,j}-u_{ij}}{h} \right|^2 + \sum_{i=0}^{p}\sum_{j=0}^{p-1} a_{i,j+1/2} \left| \frac{u_{i,j+1}-u_{ij}}{h} \right|^2 } \,,
\]
and define the relative errors accordingly by
\[
\mbox{\rm relative $L^2$ error} = \frac{\|u_{\text{ref}}-u_{\text{approx}}\|_{L^2}}{\|u_{\text{ref}}\|_{L^2}}\,, \quad
\mbox{\rm relative energy error} = \frac{\|u_{\text{ref}}-u_{\text{approx}}\|_{\mathcal{E}}}{\|u_{\text{ref}}\|_{\mathcal{E}}}.
\]

We first describe numerical experience with the offline stage. Each
interior patch $\Omega_m$ is enlarged by a margin $\Delta
x_{\text{b}}$ to damp the boundary effects. The resulting buffered
patch $\wt{\Omega}_m$ is concentric with $\Omega_m$; see
\Cref{fig:elliptic_num_setting}. In the plots shown below, we study
the patch indexed by $m = (2,2)$.

\begin{figure}[htbp]
  \centering
  \includegraphics[width=0.5\textwidth]{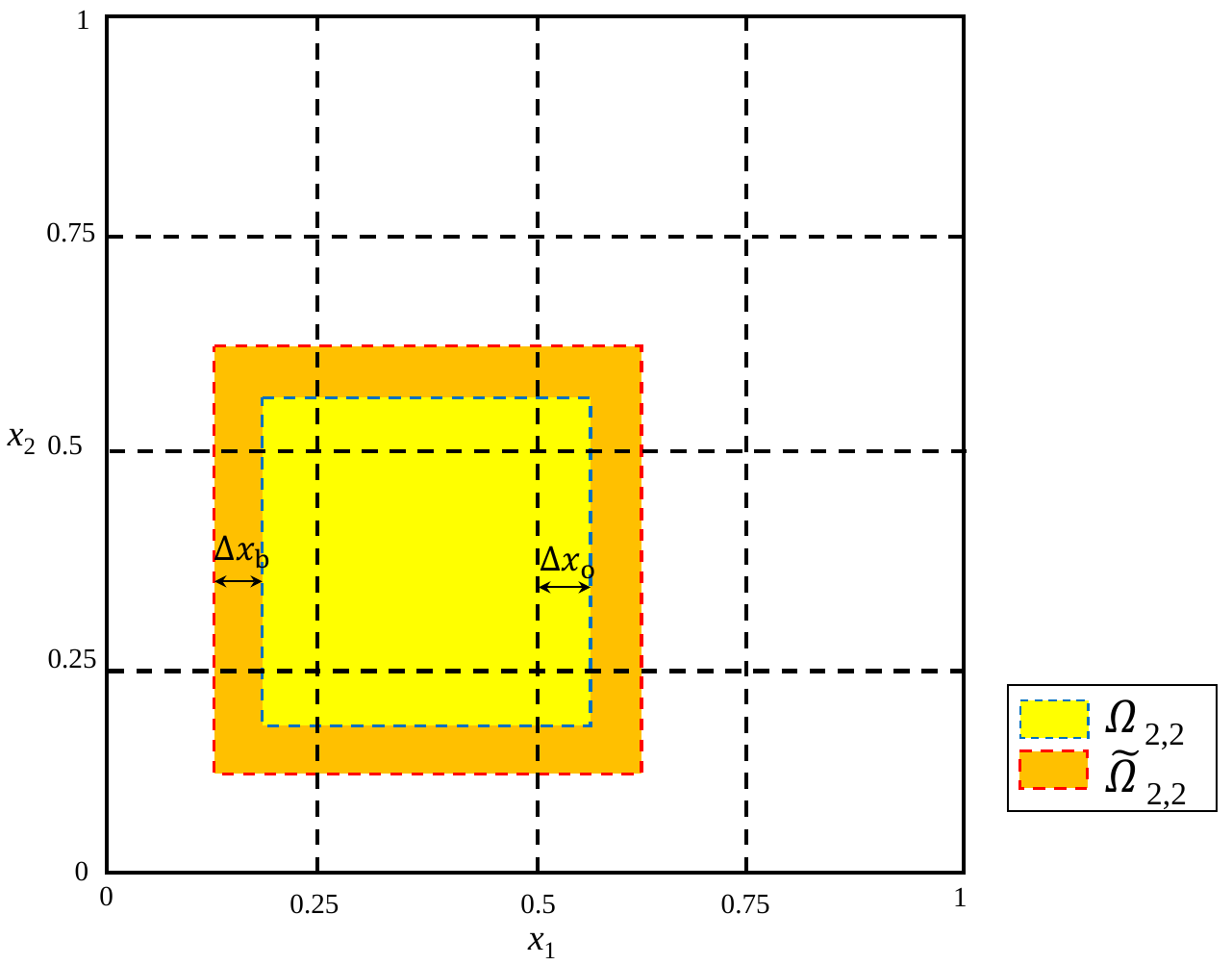}
  \caption{Buffered domain decomposition.}
  \label{fig:elliptic_num_setting}
\end{figure}

To build the local dictionary, we generate $64$ samples randomly in
$B(R_{2,2},\wt{X}_{2,2})$, where $R_{2,2} = 20$. (The sampling scheme
is discussed in~\Cref{app:elliptic}.) We compute the local solutions
with these boundary conditions on $\wt{\Omega}_{2,2}$, for several
choices of buffer size $\Delta x_{\text{b}}$, and subtract the
solutions from the reference solution, confined to
$\Omega_{2,2}$. This procedure forms the tangent space centered around
the reference solution in this particular patch. In
\Cref{fig:elliptic_tan_vec_svd} we plot the singular value decay of
this tangent space, for $\eps = 2^{-4}$. It is clear that the singular
values decay exponentially, with a larger buffer margin $\Delta
x_{\text{b}}$ leading to a faster decay rate. This observation
suggests that the tangent space is approximately low dimensional. We
then project the reference solution onto the space spanned by its
closest neighbors. As the number of neighbors increases, the relative
error decays exponentially, as seen
in~\cref{fig:elliptic_patch_proj}. When the buffer margin is $\Delta
x_{\text{b}} = 2^{-4}$, we achieve $99\%$ accuracy with $30$
neighbors. By comparison, the degrees of freedom for this patch is
determined by the total number of grid points on the boundary of this
patch --- 768, in this particular case.

\begin{figure}[htbp]
  \centering
  \subfloat[]{
  \includegraphics[width=0.47\textwidth]{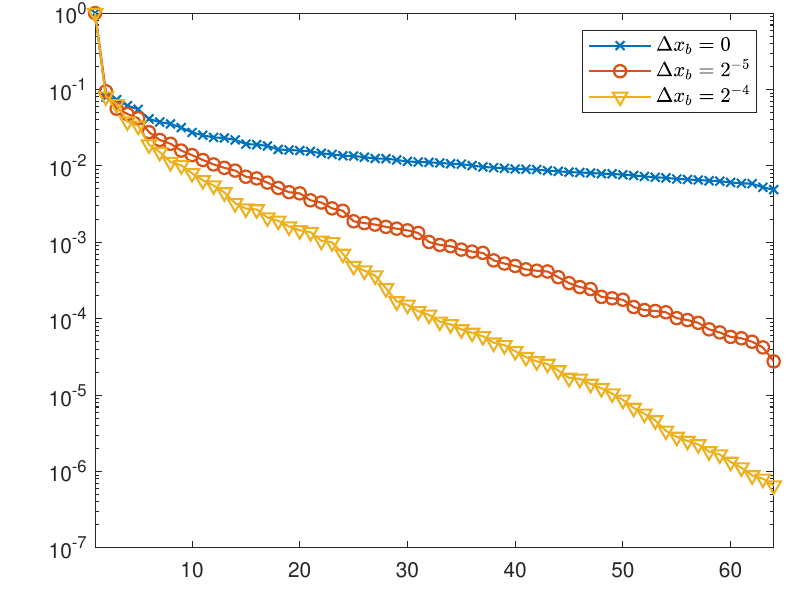} \label{fig:elliptic_tan_vec_svd}
  }
  \subfloat[]{
  \includegraphics[width=0.47\textwidth]{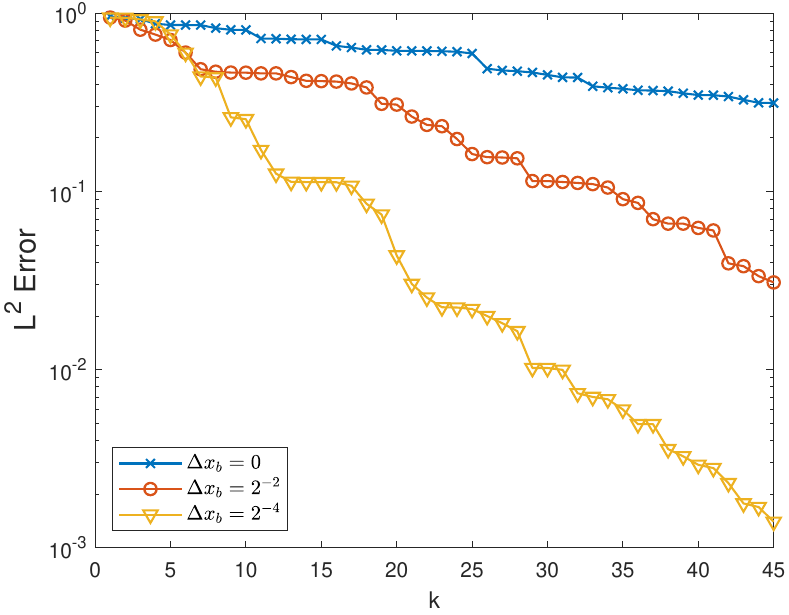} \label{fig:elliptic_patch_proj}
  }
  \caption{(a) The singular value decay of the tangent space (centered
    around the reference solution) on patch $\Omega_{2,2}$, for
    different values of the buffer margin $\Delta x_{\text{b}}$. (b)
    The relative error of the projection of reference solution onto
    the space spanned by the nearest $k$ neighbors on
    $\Omega_{2,2}$. The distance is measured in
    $L^2(\Omega_{2,2})$. $\eps = 2^{-4}$ in both plots.}
\end{figure}

In the online stage, we set the stopping criterion to be
\[
\sum_{m}\|\phi_{m}^{(n)}-\phi_{m}^{(n-1)}\|_{L^2(\partial\Omega_m)}<10^{-5}\,,
\]
where the upper index $(n)$ indicates the evaluation of the solution
in the $n$-th iteration on $\mathcal{X}_m$, which is the boundary of
$\Omega_m$. The initial guess for all local boundary condition is
chosen (trivially) to be
$\phi_{m}^{(0)}|_{\partial\Omega_{m}\backslash\partial\Omega} = 0$ and
$\phi_{m}^{(0)}|_{\partial\Omega_{m}\cap\partial\Omega} =
\phi|_{\partial\Omega_{m}\cap\partial\Omega}$.

\begin{figure}[htbp]
  \centering
  \includegraphics[width=0.32\textwidth]{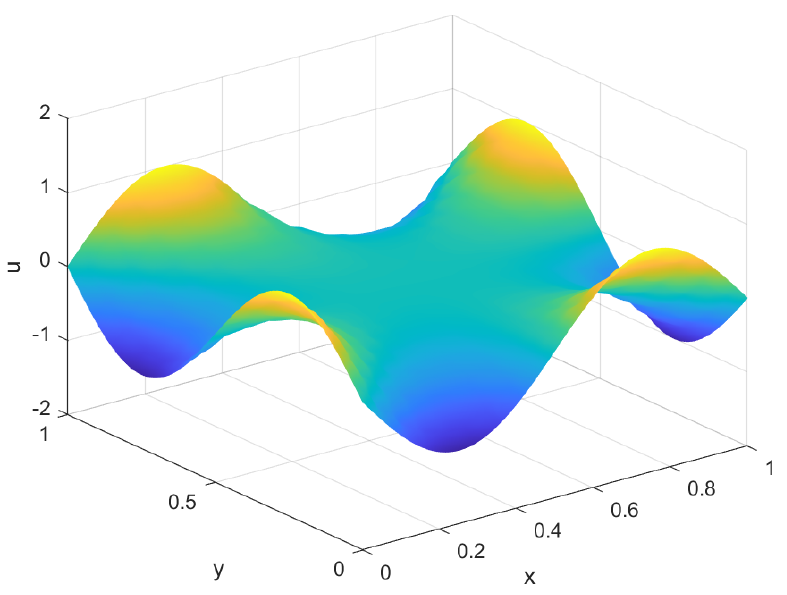}
  \includegraphics[width=0.32\textwidth]{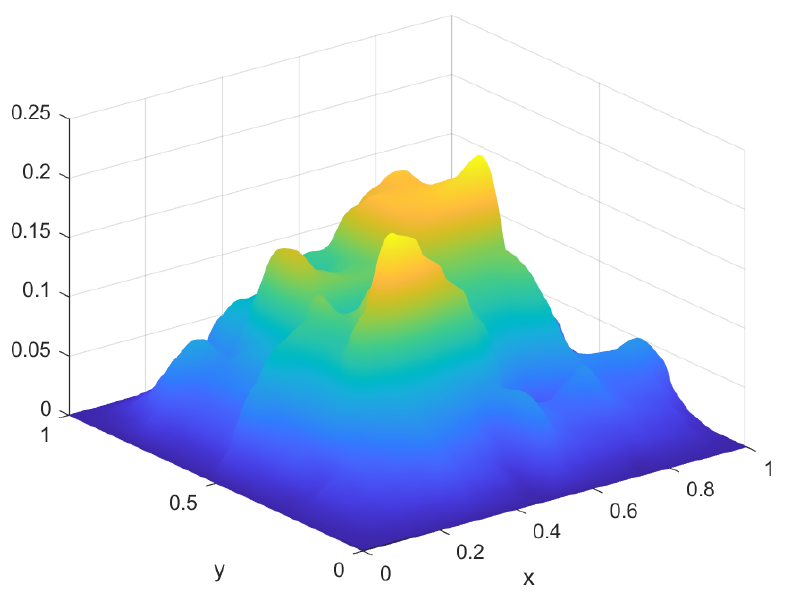}
  \includegraphics[width=0.32\textwidth]{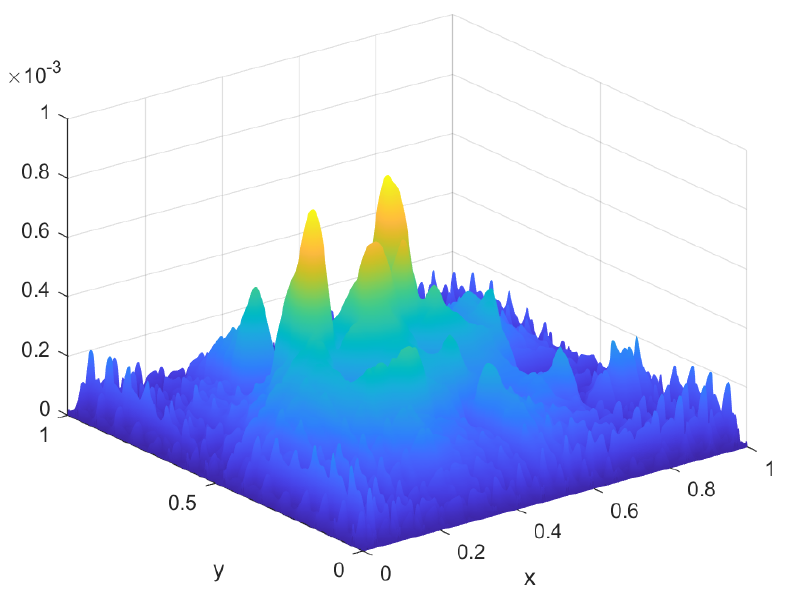}
  \caption{Computed solutions. Left panel shows the reference solution
    obtained with fine grids of width $h = 2^{-12}$. Middle and
    right panels show the numerical error $|u-u_{\text{ref}}|$ obtained
    with $k = 5$ and $k = 30$, respectively.}
  \label{fig:ref_approx_elliptic}
\end{figure}

In \Cref{fig:ref_approx_elliptic}, we compare the numerical solutions
using the space spanned by $k=5$ and $k=40$ nearest neighbors. The
buffer margin is $\Delta x_{\text{b}} = 2^{-4}$, and we set $\eps =
2^{-4}$.  We also document the error behavior as a function of $k$,
$\eps$, and $\Delta x_{\text{b}}$. In \Cref{fig:elliptic_error_eps},
we plot the error decay as a function of $k$ (the number of neighbors
used in the online stage) for different values of $\eps$ and $\Delta
x_{\text{b}}$. The decay is independent of $\eps$, indicating the rank
structure is not influenced by small scales in the equation. As the
number of neighbors $k$ increases, the global relative $L^2$ and energy error
decays exponentially provided a buffer zone is present.  When $\Delta
x_{\text{b}}=0$ (no buffer), the boundary layer effect is strong, and
convergence is not obtained, meaning that the local solution cannot be
well approximated from the dictionary.

\begin{figure}[htbp]
  \centering
  \includegraphics[width=0.3\textwidth]{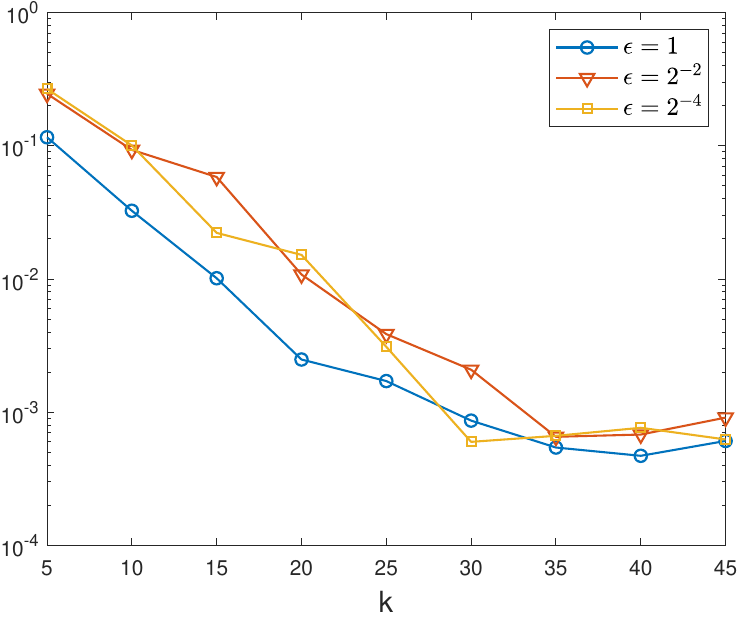}
  \includegraphics[width=0.3\textwidth]{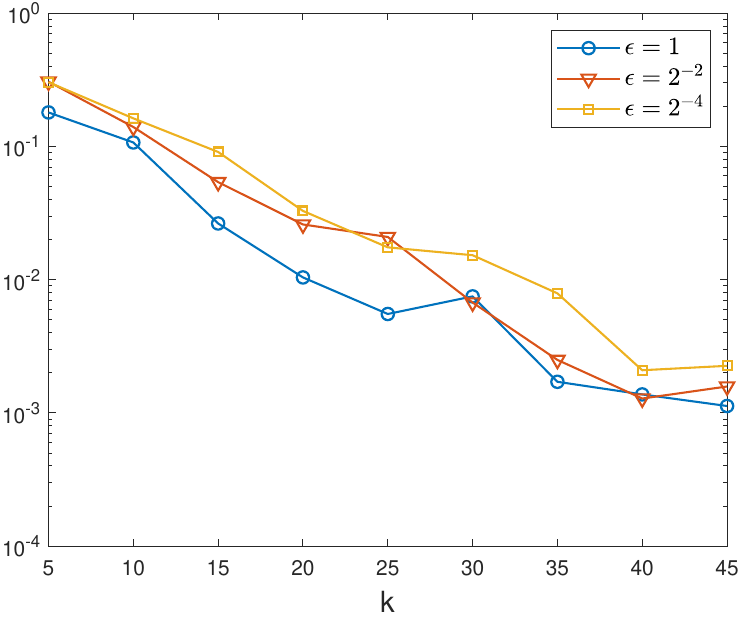}
  \includegraphics[width=0.3\textwidth]{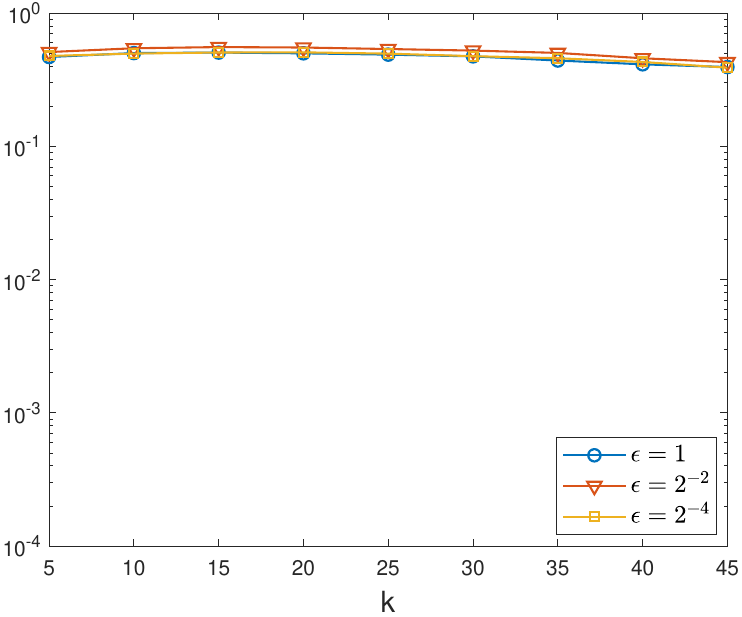}\\
  \includegraphics[width=0.3\textwidth]{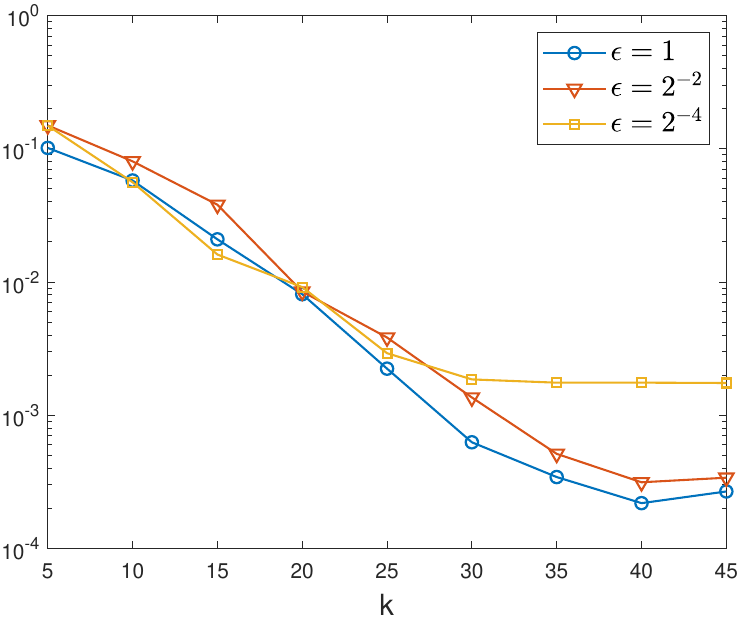}
  \includegraphics[width=0.3\textwidth]{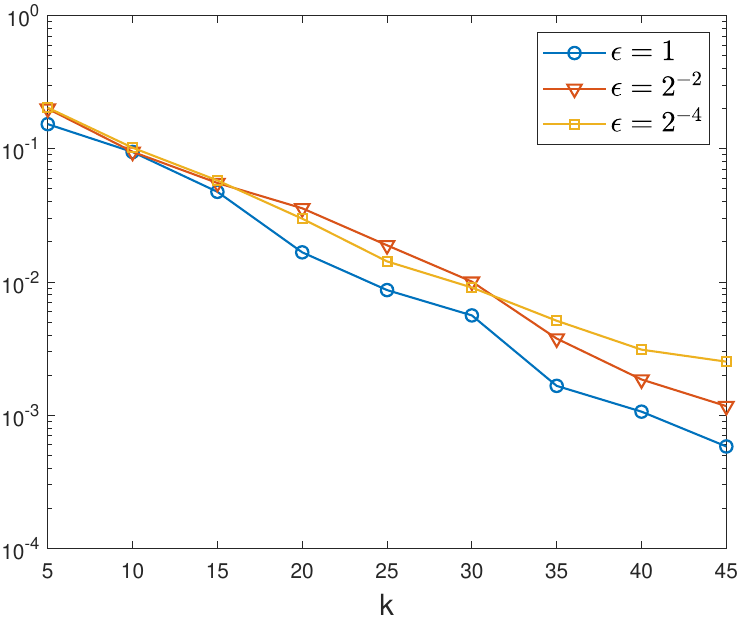}
  \includegraphics[width=0.3\textwidth]{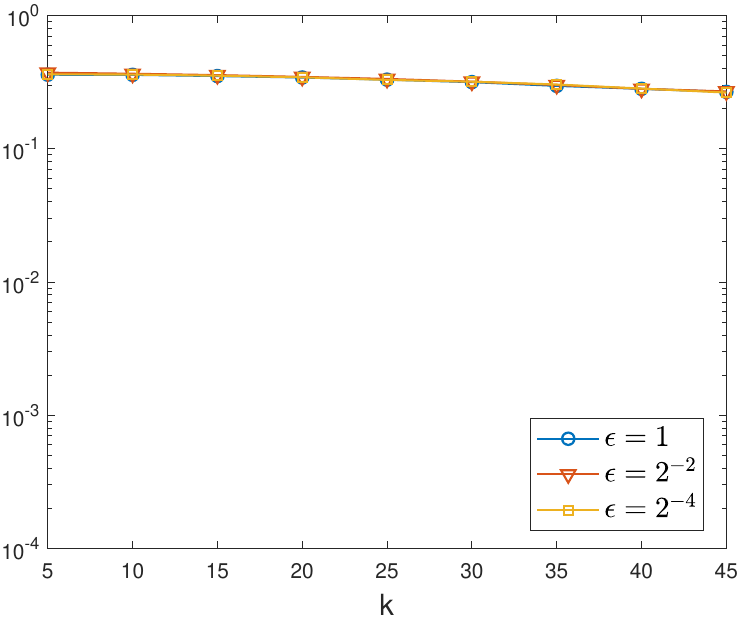}
  \caption{The top row of plots shows the global $L^2$ error as a function of $k$ with different $\eps$ and buffer zone size $\Delta x_{\text{b}}$. The bottom row of plots shows the global energy error. The three columns of plots represent  $\Delta x_{\text{b}}=2^{-4},2^{-5},0$,  respectively.}\label{fig:elliptic_error_eps}
\end{figure}

We show CPU times in Table~\ref{tbl:time_elliptic}, comparing the
reduced model for different values of $k$ with the classical Schwarz
iteration, for $\eps = 2^{-4}$ and $\Delta x_{\text{b}}=2^{-4}$. The
same stopping criterion is used for all variants. The online stage of
each reduced model is significantly faster than the classical Schwarz
iteration. Even with $k=40$ neighbors involved in the local solution
reconstruction, our method requires $1.12$s, compared to $187.8$s
required by the classical Schwarz method. While the offline
preparation is expensive in general, it is still cheaper in this
example than the classical Schwarz iteration for solving a single
problem. Because the dictionary can be reused, our method has a strong
advantage in situations where many solutions corresponding to
different boundary conditions are needed. {This is a typical situation
  in inverse problems, where to determine the unknown media, many
  boundary configurations are imposed and numerical solutions are
  computed to compare with measurements~\cite{ChLiLi:2018}.}

\begin{table}
	\centering
	\begin{tabular}{ l | c | c }
		\hline \hline
							& \multicolumn{2}{c}{CPU Time (s) ($\eps = 2^{-4}$)}\\
		\hline
		& offline & online \\
		\hline
		Reduced model $k=5$  	& 135.6914  &	0.173712		 \\
		Reduced model $k=10$	& 		    &	0.305707	     \\
		Reduced model $k=20$	& 	        &	0.462857		 \\
		Reduced model $k=30$	&    	    &	0.696785		 \\
		Reduced model $k=40$	& 	        &	1.124082		 \\
		\hline
		Classical Schwarz       & |         & 187.7705           \\
		\hline\hline
	\end{tabular}
    \caption{CPU time comparison between our reduced method with $k = 5,10,20,30,40$ and classical Schwarz method.}
  \label{tbl:time_elliptic}
\end{table}

\section{Example 2: Nonlinear radiative transfer equation}\label{sec:LowRank_RTE}

Here we study the application of Algorithm~\ref{alg:general} to a
nonlinear radiative transfer equation. Radiative transfer is the
physical phenomenon of energy transfer in the form of electromagnetic
radiation, and the radiative transfer equations describe the
absorption or scattering of radiation as it propagates through a
medium. The equations are important in optics, astrophysics,
atmospheric science, remote sensing~\cite{Mo:2013}, and other
applications.

We denote by $I^\eps(x,v)$ the distribution function of photon
particles at location $x$ moving with velocity $v$ in the physical
domain $\mathcal{D}\subset\Rbb^3$ and the velocity domain $\mathcal{V}
= \mathbb{S}^{2}$. Also denote by $T^\eps(x)$ the temperature profile
across domain $\mathcal{D}$. We consider a nonlinear system of
equations that couples the photon particle distribution with the
temperature profile.  The steady state equations are
\begin{equation}\label{eqn:nrte}
\begin{cases}
\eps v \cdot \nabla_x I^\eps = B(T^\eps) - I^\eps,\;\;\text{for } (x,v)\in \Kcal = \mathcal{D}\times \mathcal{V}, \\
\eps^2 \Delta_x T^\eps = B(T^\eps) - \langle I^\eps\rangle,\;\; \text{for } x\in \mathcal{D}\,,
\end{cases}
\end{equation}
with the velocity-averaged intensity given by
\begin{equation} \label{eq:sigvai}
\langle I\rangle(x) = \int_{\mathcal{V}}I(x,v)\mathrm{d}\mu(v)\,.
\end{equation}
Here, $\mu(v)$ is a normalized uniform measure on $\mathcal{V}$ and
$B(T)$ is a nonlinear function of $T$, typically defined as
\begin{equation} \label{eq:defB}
  B(T) = \sigma T^4,
\end{equation}
where $\sigma$ is a scattering
coefficient~\cite{KlSi:1998,ViAn:1975}. The parameter $\eps$ is called
the Knudsen number, standing for the ratio of the mean free path and
the typical domain length. When the medium is highly scattering and
optically thick, the mean free path is small, with $\eps\ll 1$. The
scattering coefficient $\sigma$ is independent of $\eps$.

We consider a slab geometry. Assuming the $y$ and $z$ directions to be
homogeneous, then since
$v=(\cos\theta,\sin\theta\sin\varphi,\sin\theta\cos\varphi)$, the
$v_x$ component becomes $\cos\theta\in[-1,1]$. The problem is
simplified to:
\begin{equation}\label{eqn:nrte_1d}
\begin{cases}
\eps v \partial_x I^\eps = B(T^\eps) - I^\eps\\
\eps^2 \partial_x^2 T^\eps = B(T^\eps) - \langle I^\eps\rangle
\end{cases}
,\quad (x,v)\in \Kcal = [a,b]\times [-1,1]\,,
\end{equation}
with $\langle I\rangle(x) = \frac{1}{2}\int_{-1}^1 I(x,v)\mathrm{d}v$.

We provide incoming boundary conditions that specify the distribution
of photons entering the domain. The boundary condition itself has no
$\eps$ dependence; we have
\begin{equation}\label{eqn:boundary_rte}
I^\eps(x,v) = I_b(x,v) \;\; \text{ on } \Gamma_-\,,  \quad T^\eps(x) = T_b(x) \;\; \text{ on } \partial\mathcal{D}\,.
\end{equation}
Here $\Gamma_\pm$ collect the coordinates at the boundary with velocity pointing into or out of the domain:
\[
\Gamma_\pm = \{(x,v): x\in\partial\mathcal{D}\,, \pm v\cdot n_x>0\}\,,
\]
and $n_x$ denotes the unit outer normal vector at
$x\in\partial\Omega$.

\subsection{Homogenization limit}
The equations~\cref{eqn:nrte} have a homogenization limit. As
$\eps\to0$, the right hand side of the equations dominates, and by
balancing the scales we obtain
\[
I^\eps\sim\langle I^\eps\rangle\sim\sigma (T^\eps)^4 \sim I^\ast\sim \sigma (T^\ast)^4\,.
\]
To find the equation satisfied by $T^\ast$, we expand the
equations~\cref{eqn:nrte} up to second order in $\eps$. Rigorous
results are shown
in~\cite{LiSu:2019,KlSc:2001,BaGoPe:1987,BaGoPeSe:1988}.
We cite
the following theorem captures the results needed here.
\begin{theorem}[Modification of Theorem~3.2 in \cite{KlSc:2001}]\label{thm:klar}
   Let $\mathcal{D}\subset\Rbb^3$ be bounded and $\partial\mathcal{D}$
   be smooth.  Assuming that the boundary
   conditions~\cref{eqn:boundary_rte} are positive and that $T_b\in
   H^{1/2}(\partial\mathcal{D})\cap L^\infty(\partial\mathcal{D})$ and
   $I_b\in L^{\infty}(\Gamma_-)$, then the nonlinear radiative
   transfer equation~\cref{eqn:nrte} has a unique positive solution
   $(I^\eps,T^\eps)\in L^\infty(\mathcal{K})\times
   L^\infty(\mathcal{D})$.  If we assume further that
   $(I_b,T_b)\geq\gamma>0$ and $I_b = B(T_b)$ a.e. on $\Gamma_-$, then
   the solution in the limit as $\eps\rightarrow 0$ converges weakly
   to $(B(T^\ast),T^\ast)$, where the limiting temperature $T^\ast$ is
   the unique positive solution to the following PDE:
  \begin{equation}\label{eqn:t_nrte}
  \Delta_x (T^\ast + B(T^\ast)/3) = 0\,, \quad \text{for} \quad x\in\mathcal{D}\,,
  \end{equation}
  equipped with Dirichlet boundary data $T^\ast|_{\partial\mathcal{D}}
  = T_b$. The convergence of $T^\eps$ is in $H^1(\mathcal{D})$ weak
  and the convergence of $I^\eps$ is in $L^\infty(\mathcal{K})$
  weak-$\ast$.
\end{theorem}

\begin{remark}
Without appropriate boundary conditions $I_b = B(T_b)$, boundary
layers of width $O(\eps)$ may emerge as $\eps\to 0$.  It is
conjectured in~\cite{KlSc:2001} that the boundary layers in the
neighborhood of each point $\hat{x}\in\partial\mathcal{D}$ can be
characterized by the following one-dimensional Milne problem for
$y\in[0,\infty)$:
\begin{align*}
-(v\cdot n_{\hat{x}})\partial_y \hat{I} &= B(\hat{T}) - \hat{I}\,,\\
\partial_y^2 \hat{T} &= B(\hat{T}) - \langle \hat{I} \rangle\,,\\
\hat{T}(0) &= T_b(\hat{x}),\quad \hat{I}(0,v) = I_b(\hat{x},v),\; \text{ for }v\cdot n_{\hat{x}}<0\,,
\end{align*}
where $y = \frac{(x-\hat{x})\cdot n_{\hat{x}}}\eps$ represents a
rescaling of the layer. The solutions that are bounded at infinity are
used to form the Dirichlet boundary conditions for~\cref{eqn:t_nrte}:
At the limit as $y\to\infty$, $B(\hat{T}) = \langle\hat{I}\rangle =
\hat{I}$, and one uses $T(\hat{x}) = \hat{T}(\infty)$.
\end{remark}

According to Theorem~\ref{thm:klar}, in the zero limit of $\eps$, $I^\eps$ loses
its velocity dependence and is proportional to $(T^\eps)^4$ that
satisfies a semi-linear elliptic equation. Since the information in
the velocity domain is lost, we expect low dimensionality of the
(discretized) solution set. For the slab problem for
RTE~\cref{eqn:nrte_1d}, the number of grid points needed for a
satisfactory numerical result is $N_xN_v$, with both $N_x$ and $N_v$
scaling as $O(\frac{1}\eps)$ for numerical accuracy. Thus, for every
given configuration of boundary conditions, the numerical solution is
one data point in an $N_xN_v$-dimensional space --- a space of very
high dimension.
However, when $\eps$ is small, the solutions are approximately given
by the limiting elliptic equation~\cref{eqn:t_nrte} and the number of
grid points needed is a number $N_x^\ast$ that has no dependence on
$\eps$. This implies that the point clouds in the
$O(1/\eps^2)$-dimensional space can be essentially represented using
$O(1)$ degrees of freedom: The solution manifold is approximately low
dimensional. (Savings are even greater for problems with higher
physical / velocity dimension.)

The use of a limiting equation to speed up the computation of kinetic
equations is not new. For Boltzmann-type equations (for which RTE
serves as a typical example),
one is interested in designing algorithms that automatically
reconstruct the limiting solutions with low computational cost. The
algorithms that achieve this property are called
``asymptotic-preserving'' (AP)
methods~\cite{Ji:1999,LeMi:2008,FiJi:2010,FiJi:2011,DiPa:2012,JiLi:2013,HuJiLi:2017,LiWa:2017,DiPa:2014,De:2011},
because the asymptotic limits are preserved automatically. There are
many successful examples of AP schemes, but most of them depend
strongly on the analytical understanding of the limiting equation. The
solver of the limiting equation is built into the Boltzmann solver, to
drag the numerical solution to its macroscopic description. Such
a design scheme limits the application of AP methods significantly. Many
kinetic equations have unknown limiting behavior, making the use of AP
designs impossible. By contrast, \Cref{alg:general} does not rely on
any explicit information of the limiting equation, and is able to deal
with general kinetic equations with small scales.

\subsection{Low dimensionality of the tangent space}\label{sec:Tan_RTE}
As for the example of \Cref{sec:LowRank_elliptic}, we start by
studying some basic properties of the local solution manifold and its
tangential plane.

We first randomly pick a point $(\ol{I}^\eps\,,\ol{T}^\eps)$ on the
solution manifold around which to perform tangential
approximation. Nearby points $(I^\eps\,,T^\eps)$ are obtained by
solutions to the RTE~\cref{eqn:nrte_1d} with respect to perturbed
boundary conditions. The boundary conditions for
$(\ol{I}^\eps\,,\ol{T}^\eps)$ and $(I^\eps\,,T^\eps)$, respectively, are
\begin{equation}\label{eqn:point_cloud_bdy}
(\ol{I}^\eps|_{\Gamma_-}\,,\ol{T}^\eps|_{\partial\mathcal{D}}) = (\ol{I}_b\,,\ol{T}_b)\,,\quad (I^\eps|_{\Gamma_-}\,,T^\eps|_{\partial\mathcal{D}}) = (I_b\,,T_b)\,,
\end{equation}
and we assume close proximity, in the sense that
\begin{equation} \label{eq:7sn}
\|\ol{I}_b-I_b\|_{L^2(\Gamma_{m,-})} + \|\ol{T}_b-T_b\|_2 = O(\delta)\,.
\end{equation}
Using the notation $\delta I^\eps := I^\eps - \ol{I}^\eps$ and $\delta
T^\eps := T^\eps - \ol{T}^\eps$ for the difference of the two
solutions, we find that this difference satisfies the equations
\begin{equation}\label{eqn:delta_equation}
\begin{cases}
\eps v \partial_x \delta I^\eps = B(\ol{T}^\eps+\delta T^\eps) - B(\ol{T}^\eps) - \delta I^\eps\,,\\
\eps^2 \partial^2_x \delta T^\eps = B(\ol{T}^\eps+\delta T^\eps) - B(\ol{T}^\eps) - \langle \delta I^\eps \rangle\,,
\end{cases}
\end{equation}
with boundary conditions:
\[
\delta I^\eps|_{\Gamma_-} = \ol{I}_b - I_b\,,\quad \delta T^\eps|_{\partial\mathcal{D}} = \ol{T}_b - T_b\,.
\]

By varying $I_b$ and $T_b$ (subject to \eqref{eq:7sn}), we obtain a
list of solutions $(\delta I^\eps\,,\delta T^\eps)$ that spans the
tangent plane of the solution manifold surrounding
$(\ol{I}^\eps\,,\ol{T}^\eps)$. It will be shown below that this plane
is low dimensional. We have the following result.
\begin{theorem}
Let $(\delta I^\eps,\delta T^\eps)$
solve~\cref{eqn:delta_equation}. As $\eps\to0$, we have $(\delta
I^\eps,\delta T^\eps)\rightarrow(\delta I^\ast,\delta T^\ast)$ so that
$\delta I^\ast = \langle \delta I^\ast \rangle = B(\ol{T}^\ast +
\delta T^\ast) - B(\ol{T}^\ast)$ and $\delta T^\ast$ solves:
\begin{equation}\label{eqn:asymp}
\partial_x^2 \left[\delta T^\ast+\tfrac{1}{3}B(\ol{T}^\ast+\delta T^\ast)-\tfrac{1}{3}B(\ol{T}^\ast)\right] = 0\,.
\end{equation}
Here the reference state $\ol{T}^\ast$ solves:
\begin{equation}\label{eqn:asymp_center}
\partial_x^2 \left[\ol{T}^\ast + \tfrac{1}{3}B(\ol{T}^\ast)\right] = 0\,.
\end{equation}
Both equations are equipped with appropriate Dirichlet type boundary
conditions. Furthermore, for small $\delta$,  the leading order
equation is
\begin{equation}
\Delta_x \left[ \left(1 + \tfrac{1}{3} B'(\ol{T}^\ast) \right)\delta T^\ast \right] = 0\,.
\end{equation}
\end{theorem}
\begin{proof}
Apply~\Cref{thm:klar} (in one dimension) to the equation for
$(\ol{I}^\eps\,,\ol{T}^\eps)$ to obtain
\[
\begin{cases}
\eps v \partial_x \ol{I}^\eps = B(\ol{I}^\eps) - \ol{I}^\eps\\
\eps^2 \partial_x^2 \ol{T}^\eps = B(\ol{I}^\eps) - \langle \ol{I}^\eps\rangle,
\end{cases}
\]
and the equation~\cref{eqn:nrte_1d} for $(I^\eps\,,T^\eps)$. Together,
these equations show that $(\ol{I}^\eps\,,\ol{T}^\eps)$ converges
weakly to $(\ol{I}^\ast\,,\ol{T}^\ast)$ that
solves~\cref{eqn:asymp_center}, and also that $(I^\eps\,,T^\eps)$
converges weakly to $(I^\ast\,,T^\ast)$ that
solves~\cref{eqn:t_nrte}. Taking the difference for
$(\ol{I}^\eps\,,\ol{T}^\eps)$ and $(I^\eps\,,T^\eps)$ we find that
$(\delta I^\eps,\delta T^\eps)$ converges to $(\delta I^\ast,\delta
T^\ast)$, which solves~\cref{eqn:asymp}.
\end{proof}

In one dimension, the elliptic problem only has two degrees of
freedom, determined by the two Dirichlet boundary conditions. This
suggests that in the limit as $\eps\to0$, for relatively small
$\delta$, the tangent plane spanned by $(\delta I^\eps,\delta T^\eps)$
is asymptotically two-dimensional, and is parameterized by the two
boundary conditions for $\delta T^\eps$.  (A similar reduction holds
in higher dimensions, but we leave the implementation to future work.)

\subsection{Implementation of the algorithm}
In RTE, the domain setup needs some extra care, and we need to
re-perform partitioning. The physical boundaries are no longer the
boundaries at which the Dirichlet conditions are imposed, and the
general framework in~\Cref{sec:algorithm} for PDE with Dirichlet
boundary condition on the physical boundaries has to be changed
accordingly. For the $(1+1)$D case, we set
\[
\Kcal = \mathcal{D}\times \mathcal{V} = [0,L]\times[-1,1] \,; \text{ then } \Gamma_{-} = \{(0,v):v>0\}\cup\{(1,v):v<0\}\,,
\]
with boundary conditions
\[
I^\eps|_{\Gamma_-} = g = (g^{(1)}(0,\cdot),g^{(2)}(L,\cdot)), \quad
T^\eps(0) = \theta^{(1)}, \quad T^\eps(L) = \theta^{(2)}\,,
\]
where $g^{(1)}$ is supported only on $v>0$ while $g^{(2)}$ is
supported only on $v<0$. For notational simplicity, we write
\[
u := (I^\eps,T^\eps)\,,\quad u|_{\Gamma_-}:=\phi = (g^{(1)}(0,
\cdot),g^{(2)}(L,\cdot),\theta^{(1)},\theta^{(2)})\,.
\]

To partition the domain, we divide $\Kcal$ into $M$ overlapping patches:
\begin{equation}
\Kcal = \bigcup_{m=1}^{M}\Kcal_m, \quad \mathrm{with} \quad \Kcal_m = \mathcal{D}_m\times \mathcal{V} = [t_m,s_m]\times[-1,1],\label{eqn:decomp}
\end{equation}
where $t_m$ and $s_m$ are left and right boundaries for the $m$-th patch, satisfying
\[
0=t_1<t_2<s_1<t_3<\dots<s_{M-2}<t_{M}<s_{M-1}<s_M=L\,.
\]
The size of the $m$th patch in $x$ direction is denoted as $d_m =
t_m-s_m$. For each patch, we define the local incoming boundary
coordinates as follows:
\begin{equation}
\Gamma_{m,-} = \{(t_m,v):v>0\}\cup\{(s_m,v):v<0\}\,.
\end{equation}
See \Cref{fig:RTE_decomp} for an illustration of the configuration.

\begin{figure}[htbp]
  \centering
  \includegraphics[width=0.9\textwidth]{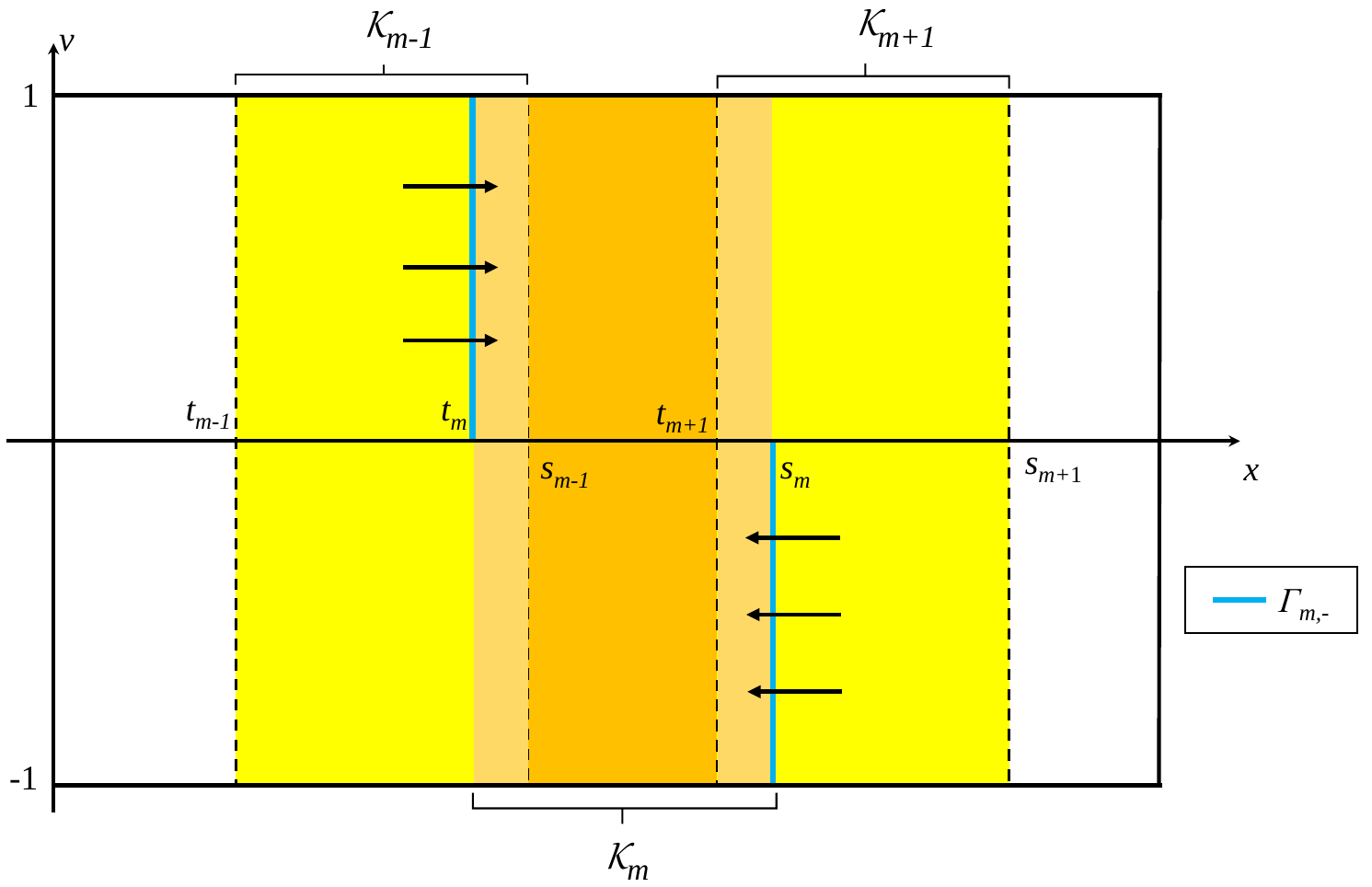}
  \caption{Domain decomposition for nonlinear RTE and the incoming boundary of the local patch.}
  \label{fig:RTE_decomp}
\end{figure}

In this particular setup, according to~\cite{LiSu:2019}, if $\phi$ is
in the space
\[
\mathcal{X} = L^2(\Gamma_{-})\times \mathbb{R}_+^2 = \left\{\left(g,\theta^{(1)},\theta^{(2)}\right)\mid g\in L^2(\Gamma_{-}); \; \theta^{(1)},\theta^{(2)}\geq0\right\}\,,
\]
then there exists a unique positive solution in the space
\[
\mathcal{Y} = H_2^1(\Kcal)\times
H^1(\mathcal{D})=\left\{(I,T)\mid I\in H_2^1(\Kcal),T\in
H^1(\mathcal{D})\right\}\,,
\]
where $H_2^1(\Kcal)$ is the space of functions for which the following
norm is finite:
\[
\|I\|_{H_2^1(\Kcal)} = \|I\|_{L^2(\Kcal)} + \|v \partial_x I\|_{L^2(\Kcal)}\,.
\]
Note that the trace operators $T_{\pm}u = u|_{\Gamma_{\pm}}$ are well-defined maps from
$H_2^1(\Kcal)$ to $L^2(\Gamma_\pm)$ (see, for
example~\cite{Ag:2012}).

To proceed, we define several operators. We denote spaces associated
with each patch $m$ as follows:
\begin{align*}
\mathcal{X}_m & := L^2(\Gamma_{m,-})\times \mathbb{R}_+^2 = \left\{\left(g,\theta^{(1)},\theta^{(2)}\right) \mid g\in L^2(\Gamma_{m,-}),\; \theta^{(1)},\theta^{(2)}\geq0\right\}\,,\\
\mathcal{Y}_m &:= H_2^1(\Kcal_m)\times H^1(\mathcal{D}_m)=\left\{(I,T)\mid I\in H_2^1(\Kcal_m), \; T\in H^1(\mathcal{D}_m)\right\}\,.
\end{align*}
Then we have the following operator definitions for each patch
$m$. (For simplicity of notation, we set $\sigma \equiv 1$ in the
definition \eqref{eq:defB} of $B(T)$.)
\begin{itemize}
  \item The solution operator $\mathcal{S}_m: \mathcal{X}_m \to
    \mathcal{Y}_m$ satisfies
  $\mathcal{S}_m \phi_m = u_m$,
  where $u_m = (I_m^\eps,T_m^\eps)$ solves the RTE on patch $\Kcal_m$
  with boundary condition \\ $\phi_m =
  (g_m,\theta_{m}^{(1)},\theta_{m}^{(1)})$:
  \[
  \begin{cases}
  \eps v \partial_x I_m^\eps &= (T_m^\eps)^4 - I_m^\eps\\
  \eps^2 \partial_x^2 T_m^\eps &= (T_m^\eps)^4 - \langle I_m^\eps \rangle
  \end{cases} \,, \;  (x,v)\in \Kcal_m\,,
  \]
  with $T_m^\eps(t_m) = \theta_m^{(1)}$, $T_m^\eps(s_m) = \theta_m^{(2)}$, and
  \[
  I_m^\eps|_{\Gamma_{m,-}} = g_m(x,v) = (g_m^{(1)}(x,v),g_m^{(2)}(x,v))\,.
  \]
  \item The restriction operator $\mathcal{I}_{m\pm1}^{m}$ from patch
    $\Kcal_m$ to the boundaries of adjacent patches, namely,
    $\Kcal_m\cap\Gamma_{m\pm1,-}$ and
    $\mathcal{D}_m\cap\partial\mathcal{D}_{m\pm1,-}$, is defined as follows:
  \begin{align*}
  &\mathcal{I}_{m+1}^{m}u_m = (I_m^\eps|_{\Kcal_m\cap\Gamma_{m+1,-}},T_m^\eps|_{\mathcal{D}_m\cap\partial\mathcal{D}_{m+1,-}}), \quad m = 1,\dotsc,M-1\,,\\
  &\mathcal{I}_{m-1}^{m}u_m = (I_m^\eps|_{\Kcal_m\cap\Gamma_{m-1,-}},T_m^\eps|_{\mathcal{D}_m\cap\partial\mathcal{D}_{m-1,-}}), \quad m = 2,\dotsc,M\,.
  \end{align*}
  \item The boundary update operator
    $\mathcal{P}_m:\mathcal{X}_{m-1}\oplus\mathcal{X}_{m+1}\rightarrow\mathcal{X}_m$
    is defined for $m \neq 1$ and $m \neq M$ by
  \begin{equation}
  \mathcal{P}_m(\phi_{m-1},\phi_{m+1}) =
  (\mathcal{I}_m^{m-1}\mathcal{S}_{m-1}\phi_{m-1},\mathcal{I}_m^{m+1}\mathcal{S}_{m+1}\phi_{m+1}).
  \end{equation}
  For the two ``end'' patches $\Kcal_1$ and $\Kcal_M$ that intersect
  with physical boundary $\Gamma_-$, boundary conditions are updated
  only in the interior of the domain:
  \begin{alignat*}{2}
  \mathcal{P}_1 &:\mathcal{X}\times\mathcal{X}_{2}\rightarrow\mathcal{X}_1,\quad & \mathcal{P}_1(\phi,\phi_{2}) & = (\phi|_{\Gamma_-\cap\Gamma_{1,-}},\mathcal{I}_1^{2}\mathcal{S}_{2}\phi_{2})\,, \\
  \mathcal{P}_M &:\mathcal{X}_{M-1}\times\mathcal{X}\rightarrow\mathcal{X}_M,\quad & \mathcal{P}_M(\phi_{M-1},\phi) &= (\mathcal{I}_M^{M-1}\mathcal{S}_{M-1}\phi_{M-1},\phi|_{\Gamma_-\cap\Gamma_{M,-}})\,.
  \end{alignat*}
\end{itemize}

As suggested by~\Cref{alg:general}, in the offline stage, we construct
local dictionaries on interior patches from a few random samples,
enlarging each interior patch slightly to eliminate the boundary layer
effect. Define $\wt{\Kcal}_m$ and $\wt{\mathcal{D}}_m$ such that
\[
\Kcal_m \subset\wt{\Kcal}_m = \wt{\mathcal{D}}_m\times \mathcal{V}\,,
\]
where $\mathcal{D}_m \subset \wt{\mathcal{D}}_m\subset\mathcal{D}$
expands the boundary of $\mathcal{D}_m$ to both sides by a margin of
$\Delta x_{\text{b}}$. Denoting by $\wt{\Gamma}_{m,-}$ the boundary
coordinates corresponding to $\wt{\mathcal{D}}_m$, we let
$\wt{\mathcal{X}}_m=L^2(\wt{\Gamma}_{m,-})\times\mathbb{R}^2$ capture
the boundary conditions on $\partial\wt{\mathcal{D}}_m$.

We draw $N$ samples $\wt{\phi}_{m,i}$, $i=1,2,\dotsc,N$, randomly from
the set
\[
B_+(R_m; \wt{\mathcal{X}}_m) := \{\wt{\phi} = (\wt{I}_B,\wt{T}_B)\in\wt{\mathcal{X}}_m: \|\wt{\phi}\|_{\wt{\mathcal{X}}_m}\leq R_m, \; \wt{I}_B\geq 0, \; \wt{T}_B\geq 0\} \,.
\]
(The sampling procedure is discussed in in~\Cref{app:nrte}.)
The local
solutions $\wt{u}_{m,i} = (\wt{I}_{m,i}^\eps,\wt{T}_{m,i}^\eps)$ solve
\begin{equation}\label{eqn:patch_learn}
\begin{aligned}
&\begin{cases}
\eps v \partial_x \wt{I}_{m,i}^\eps = (\wt{T}_{m,i}^\eps)^4 - \wt{I}_{m,i}^\eps\\
\eps^2 \partial_x^2 \wt{T}_{m,i}^\eps = (\wt{T}_{m,i}^\eps)^4 - \langle\wt{I}_{m,i}^\eps \rangle
\end{cases}
\;\; (x,v)\in \wt{\Kcal}_m\,, \\
&\quad (\wt{I}_{m,i}^\eps|_{\wt{\Gamma}_{m,-}},\wt{T}_{m,i}^\eps|_{\partial\wt{\mathcal{D}}_m}) = \wt{\phi}_{m,i}\,, \quad i =1,2,\dotsc,N.
\end{aligned}
\end{equation}
The solutions to these equations, confined to the original patch
$\Kcal_m$ and its boundary $\Gamma_{m}$, are used to construct two
dictionaries:
\begin{equation}
\mathscr{I}_m = \{\psi_{m,i}\}_{i=1}^N,\quad \mathscr{B}_m = \{\phi_{m,i}\}_{i=1}^N.
\end{equation}
where
\[
\psi_{m,i} = (\wt{I}_{m,i}^\eps|_{\Kcal_m},\wt{T}_{m,i}^\eps|_{\mathcal{D}_m})\,,\quad \phi_{m,i} = (\wt{I}_{m,i}^\eps|_{\Gamma_{m,-}},\wt{T}_{m,i}^\eps|_{\partial\mathcal{D}_m})\,.
\]

In the online stage, at each iteration, we seek neighbors to
interpolate for local solutions.
We use the $L^2$ norm to measure the distance between the newly generated solutions and the older solution set.
Denote by $\phi_m^{(n)}$ the solution
at the $n$-th iteration in patch $\Kcal_m$, and define by
\[
\{\phi_{m,i_q^{(n)}}\,, \quad q=1,2,\dotsc, k\}
\]
its $k$ nearest neighbors in $\mathscr{B}_m$, for some chosen positive
integer $k$, with the indices $i_q^{(n)}$ being ordered so that
$\phi_{m,i_1^{(n)}}$ is the nearest neighbor. Then we define the local
tangential approximation $\mathcal{S}_m\phi_m^{(n)}$ by:
\begin{equation}\label{eqn:approx_soln_op}
u_m^{(n)}=\mathcal{S}_m\phi_m^{(n)} = \psi_{m,i_1^{(n)}}+\Psi_{m}^{(n)}c_m^{(n)}\,,
\end{equation}
where $\Psi_m^{(n)}$ and $c_m^{(n)}$ are defined as
in~\cref{eqn:tan_space_general} and~\cref{eqn:coef_general}. The local
solution is then updated as follows:
\begin{equation}\label{eqn:patch2}
\phi_{m}^{(n+1)} = \mathcal{P}_m(\phi_{m-1}^{(n)},\phi_{m+1}^{(n)}) = (\mathcal{I}_m^{m-1}\mathcal{S}_{m-1}\phi_{m-1}^{(n)},\mathcal{I}_m^{m+1}\mathcal{S}_{m+1}\phi_{m+1}^{(n)})\,.
\end{equation}
For $m = 1$ and $m= M$, to avoid updating the physical boundary, we
set
\begin{align*}
\mathcal{P}_1(\phi,\phi_{2}^{(n)}) & = (\phi|_{\Gamma_-\cap\Gamma_{1,-}},\mathcal{I}_1^{2}\mathcal{S}_{2}\phi_{2}^{(n)})\,, \\
\mathcal{P}_M(\phi_{M-1}^{(n)},\phi) & = (\mathcal{I}_M^{M-1}\mathcal{S}_{M-1}\phi_{M-1}^{(n)},\phi|_{\Gamma_-\cap\Gamma_{M,-}})\,.
\end{align*}

Once the convergence is achieved (at iteration $n$, say), we assemble
the final solution as
\begin{equation}\label{eqn:global_soln}
u_\text{final} = u^{(n)} = \sum_{m=1}^M \chi_m u_m^{(n)}\,,
\end{equation}
with $\chi_m:\Omega\rightarrow\mathbb{R}$ being the smooth partition of unity
associated with the partition of $\Kcal$.

\subsection{Numerical Tests}\label{sec:num_rte}
In the numerical tests, we take the domain to be
\[
\Kcal = \mathcal{D}\times\mathcal{V} = [0,L] \times [-1,1] =
      [0,3]\times[-1,1].
\]
To form the patch $\Kcal_m = \mathcal{D}_m\times\mathcal{V}$, the
domain $\mathcal{D}$ is divided into $M = 7$ non-overlapping patches
whose widths are $d_1 = d_7 = \tfrac{L}{2(M-1)} = 0.25$ and $d_i =
\tfrac{L}{M-1} = 0.5$, $i = 2,\cdots,M-1$. Each patch is then enlarged
by $\Delta x_{\text{o}} = .125$ to both sides (except the ones
adjacent to the physical boundary, which are enlarged only on the
``internal'' sides), so we have
\[
\begin{aligned}
&\mathcal{D}_m = \left(\tfrac{L(2m-1)}{2(M-1)}-\Delta x_{\text{o}},\tfrac{L(2m-1)}{M}+\Delta x_{\text{o}}\right)\,, \;\; m=2,\dots,M-1\,,\\
&\mathcal{D}_1 = \left(0,\tfrac{L}{2(M-1)}+\Delta x_{\text{o}}\right)\,,\quad \mathcal{D}_M = \left(L-\tfrac{3}{2(M-1)}-\Delta x_{\text{o}},3\right)\,.
\end{aligned}
\]
The region of overlap between adjacent patches $\Kcal_m$ has size
$2\Delta x_{\text{o}}\times [-1,1]$.
The partition of unity functions over each patch $\Kcal_m$ are obtained using the method of~\Cref{sec:num_elliptic}

Denote the spatial grid points by $0 = x_0 < x_1 < \cdots < x_{N_x-1}
< x_{N_x} = L$, which is a uniform grid with step size $\Delta x =
\tfrac{L}{N_x}$.  The velocity grid points are denoted by
$-1<v_1<v_2<\cdots<v_{N_v-1}<v_{N_v}<1$ for some even value of
$N_v$. We use the Gauss-Legendre quadrature points for the $v_i$.  The
numerical solutions are denoted by $I^{ij} \approx I(x_i,v_j)$ and
$T^i \approx T(x_i)$. To quantify the numerical error, we denote the
discrete $L^2$ norm of $u = ([I^{ij}],[T^{i}])$ by
\begin{align*}
\|u\|^2_2 &= \sum_{j=1}^{N_v} w_j \frac{\Delta x}{2}  |I^{0j}|^2 + \sum_{j=1}^{N_v} w_j \frac{\Delta x}{2}  |I^{iN_x}|^2 + \sum_{i = 1}^{N_x-1}\sum_{j=1}^{N_v} w_j \Delta x  |I^{ij}|^2 \\
& \quad + \frac{\Delta x}{2} |T^0|^2 + \frac{\Delta x}{2} |T^{N_x}|^2 + \sum_{i = 1}^{N_x-1} \Delta x |T^i|^2 \,,
\end{align*}
where $w_j$ is the Gauss-Legendre weight, and the relative error
$u_{\text{ref}}$ between a reference solution and an approximate
solution $u_{\text{approx}}$ is defined by
\[
\text{relative } L^2 \text{ error} = \frac{\|u_{\text{ref}}-u_{\text{approx}}\|_2}{\|u_{\text{ref}}\|_2}\,.
\]

We solve the PDE using finite differences. The intensity equation is
discretized in space by a classical second-order exponential finite
difference scheme~\cite{Il:1969,RoStTo:2008}, and the temperature
equation is approximated by the standard three-point scheme. The
resulting nonlinear system is then solved by fixed point
iteration~\cite{KlSc:2001,LiSu:2019}, where in each evaluation of the
fixed point map, the monotone iterative method is exploited to solve
the semilinear elliptic equation. For computations with $\eps =
2^{-4}$ and $\eps = 2^{-6}$, we further use Anderson acceleration to
boost the convergence of fixed point
iteration~\cite{An:1965,FaSa:2009,WaNi:2011}.

We use extremely fine discretization with
$\Delta x = 2^{-14} = \tfrac{1}{16384}$ and $N_v =
 2^{10} = 1024$. The discretization is fine enough for
us to view it as the reference solution.
All the other computations are done with coarser mesh $\Delta x =
2^{-11} = \tfrac{1}{2048}$ and $N_v = 2^{7} = 128$.

The boundary condition
$\phi=(g^{(1)},g^{(2)},\theta^{(1)},\theta^{(2)})$ is defined as
follows:
\begin{align*}
g^{(1)}(0,v>0) = 3+\sin(2\pi v)\,,& \quad g^{(2)}(L=3,v<0) = 2+\sin(2\pi v)\,,\\
   \theta(0) = \theta^{(1)} = 2 \,,& \quad \theta(L) = \theta^{(2)} = 3 \,. \\
\end{align*}

The enlarged patches needed in the offline stage, denoted by
$\wt{\Kcal}_m$, are obtained by enlarging each respective $\Kcal_m$ by
the quantity $\Delta x_{\text{b}}$.
The configuration of the domain and the partition are seen in
\Cref{fig:decomp_numerical}, where $\Delta x_{\text{b}} = .125$.

\begin{figure}[htbp]
  \centering
  \includegraphics[width=0.7\textwidth]{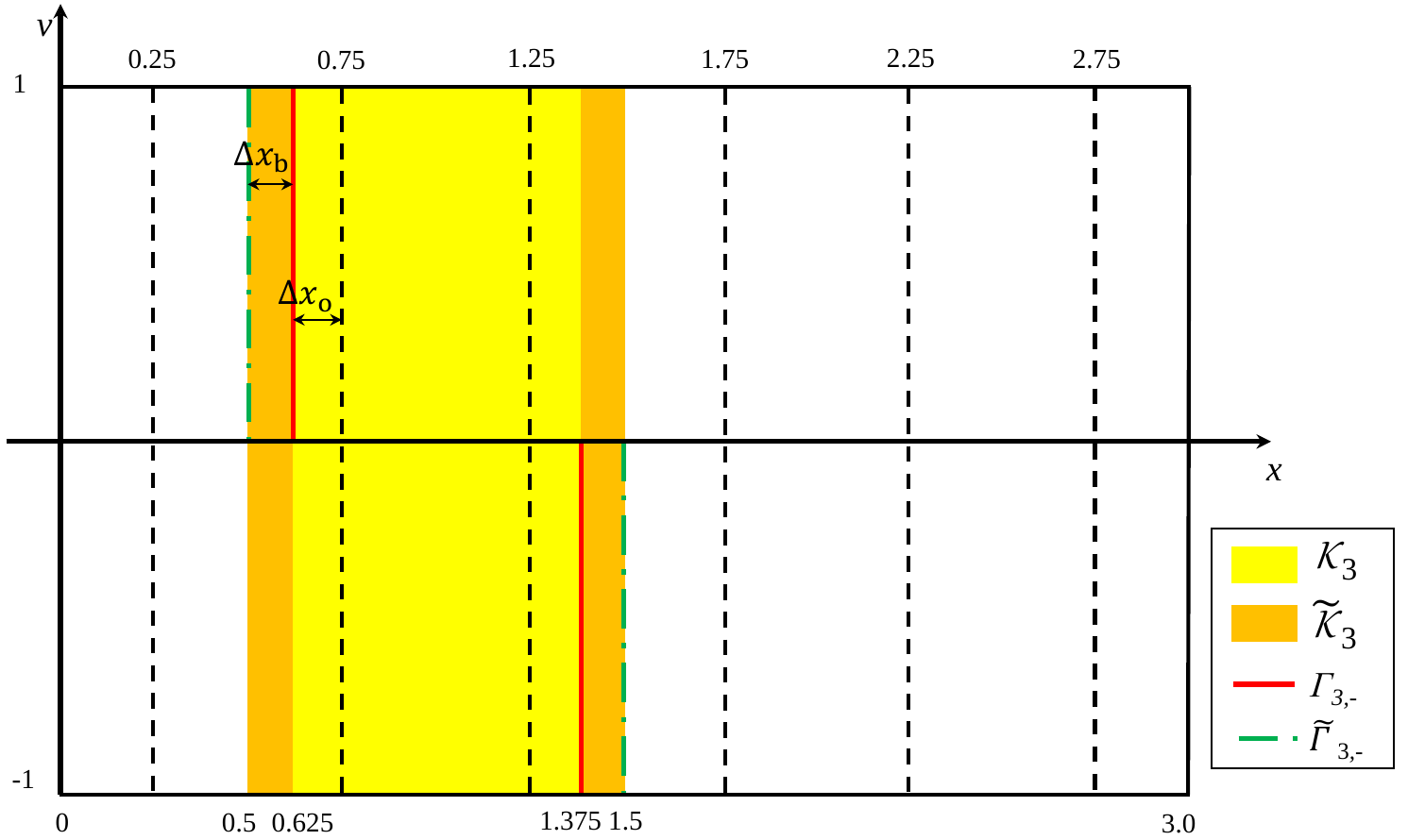}
  \caption{Configuration of patches (including enlarged patches) in the decomposed domain}
  \label{fig:decomp_numerical}
\end{figure}

On the buffered interior patch $\wt{\Kcal}_m$, we sample $N = 64$
configuration of boundary conditions in
$B_+(R_m;\wt{\mathcal{X}}_m)$. On the discrete level, this process
finds $64$ boundary conditions $\wt{\phi}$ so that
\[
\|\wt{\phi}\|^2 = \sum_{j=1}^{\frac{N_v}{2}} w_j
|\wt{g}^{(2)}(s,v_j)|^2 + \sum_{j=\frac{N_v}{2}+1}^{N_v} w_j
|\wt{g}^{(1)}(t,v_j)|^2 + |\wt{\theta}^{(1)}|^2 +
|\wt{\theta}^{(2)}|^2<R_m\,.
\]
We set $R_m = 25$ in our experiments.

To demonstrate the linearity of the updating map $\mathcal{P}_m$, we
choose the patch $\Kcal_3=[0.625,1.375]\times[-1,1]$, which overlaps
$\Kcal_2$ at $[0.625,0.875]\times[-1,1]$. For $\Delta x_{\text{b}} =
2^{-3}$ and $\eps = 2^{-6}$, we compute local solutions on the
buffered domain $\wt{\Kcal}_3$ with $64$ different configurations, and
evaluate $T$ at $0.625$ and $1.375$ (the two ending points of
$\Kcal_3$) and at $0.875$ (the point that intersects with
$\partial\Kcal_2$). In \Cref{fig:offline}, we plot $T(0.875)$ as a
function of $T(0.625)$ and $T(1.375)$. We observe that it is a slowly
varying two-dimensional manifold and is locally almost linear. Thus,
$T(0.875)$ can be determined uniquely by the pair of values
$(T(0.625),T(1.375))$.
Further, we plot $\tfrac{\langle |I(x,\cdot) - \langle
  I\rangle(x)|^2\rangle}{\langle I\rangle(x)^2}$ and $\tfrac{\langle
  I\rangle(x) - T^4(x)}{T^4(x)}$ at $x=0.875$, showing that the
relative variation is nearly zero. This means that $I$ is essentially
constant at $x=0.875$, with $I = T^4$. These calculations suggest that the
entire solution on this patch is uniquely determined by $T(0.625)$ and
$T(1.375)$, implying that the local degrees of freedom for the solution in
the entire patch is only two, so that the local solution manifold is
approximately two-dimensional.

\begin{figure}[htbp]
  \centering
  \includegraphics[width=0.32\textwidth]{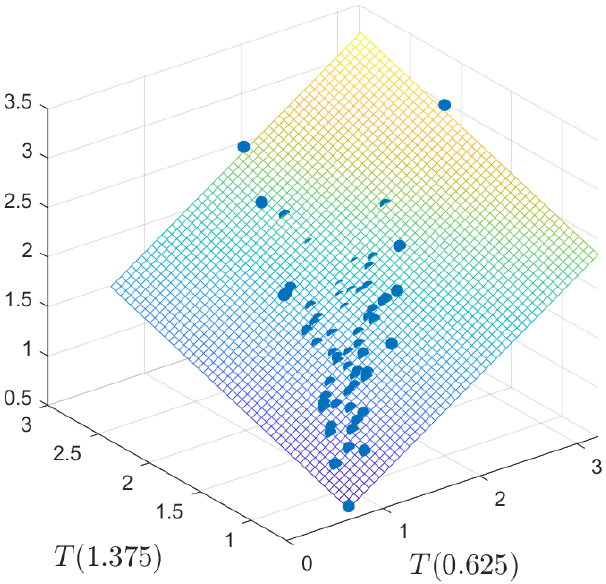}
  \includegraphics[width=0.32\textwidth]{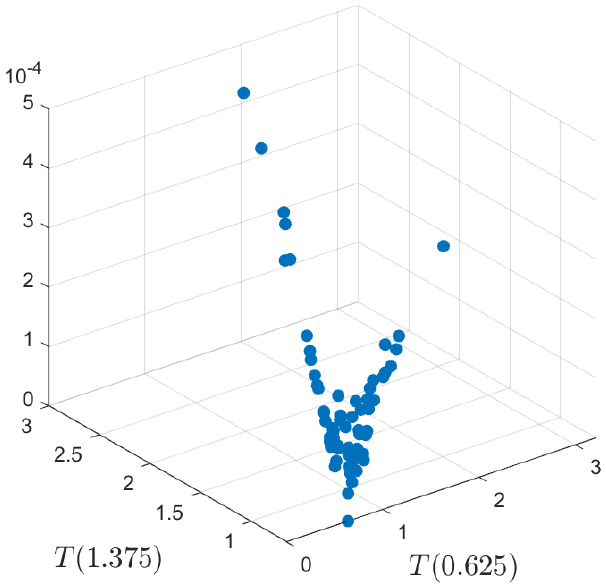}
  \includegraphics[width=0.32\textwidth]{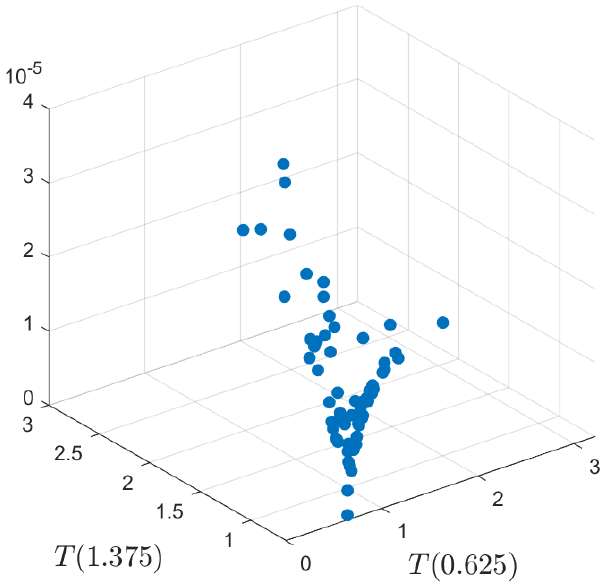}
  \caption{The plot on the left shows the point cloud
    $(T(0.625),T(1.375),T(0.875))$ and its fitting plane. We observe
    that the manifold is approximately two-dimensional, so that
    $T(0.875)$ can be uniquely determined by
    $(T(0.625),T(1.375))$. The middle and right panels show the
    quantities $\tfrac{\langle |I(x,\cdot) - \langle
      I\rangle(x)|^2\rangle}{\langle I\rangle(x)^2}$ and
    $\tfrac{\langle I\rangle(x) - T(x)^4}{T(x)^4}$ at $x=0.875$,
    respectively, showing that the solution is nearly constant, with $I=T^4$.}
  \label{fig:offline}
\end{figure}

To verify that the local dictionary represents the solution manifold
adequately, we confine the reference solution in patch $\Kcal_2$ and
project it onto the space spanned by its nearest $k$ modes in the
local dictionary. We evaluate the resulting relative error as a
function of $k$, plotting the result in \Cref{fig:error_proj}. For
$\eps = 2^{-6}$ and $\Delta x_{\text{b}} = .125$, we observe a sharp
decay of error when $k\geq 3$, meaning that the local reference
solution can be represented to acceptable accuracy by two local
dictionary modes, and suggesting once again that the local solution
manifold is two-dimensional.
\begin{figure}[htbp]
  \centering
  \subfloat[$\Delta x_{\text{b}} = 2^{-2}$]{
  \includegraphics[width=0.45\textwidth]{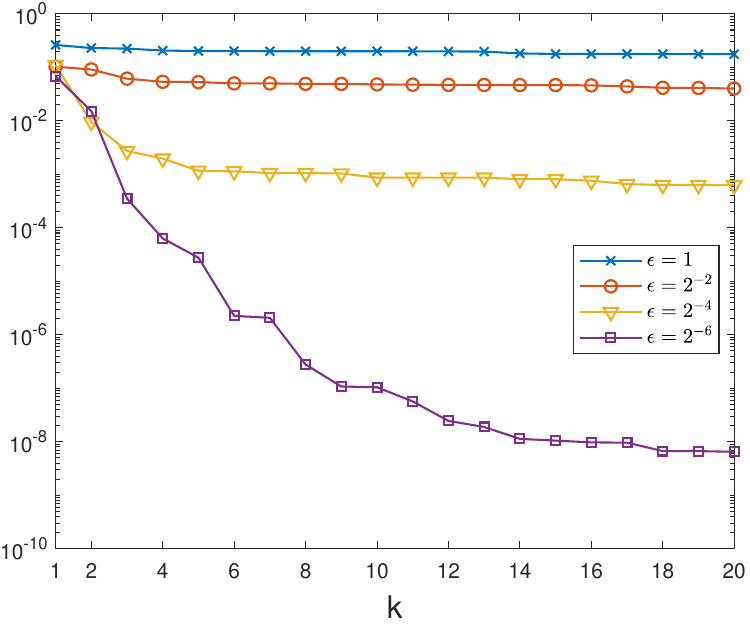}
  }
  \subfloat[$\Delta x_{\text{b}} = 2^{-3}$]{
  \includegraphics[width=0.45\textwidth]{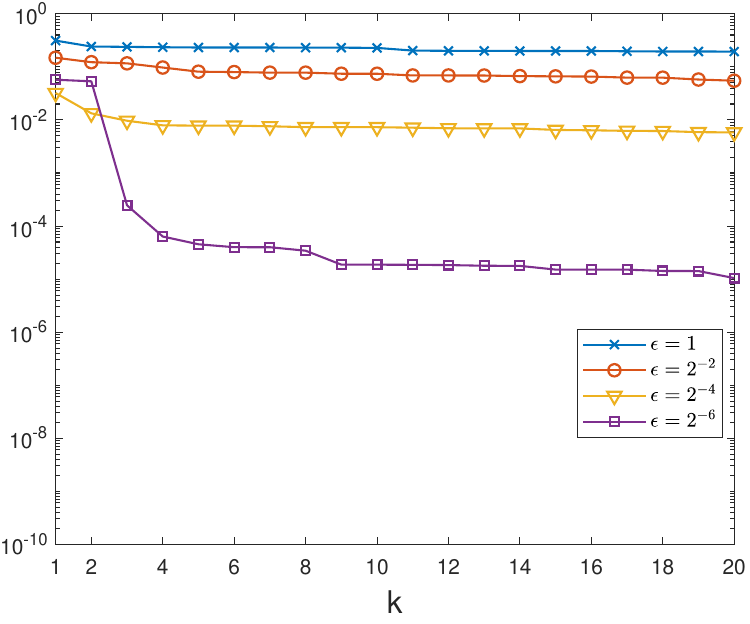}
  }
  \caption{The relative error of the $L^2$ projection of the reference solution onto the space spanned by the nearest $k$ modes on the patch $\Kcal_2$.}
  \label{fig:error_proj}
\end{figure}

The sample number $N$ and the radius $R_m$ are two crucial parameters that affect the effectiveness of the method. We check how the approximation capability of the local dictionary depends on the two parameters over the local patch $\Kcal_2$. In~\Cref{fig:error_vs_N}, we show the projection error as $N$ increases for different $R_m$. The error of the dictionary saturates as $N$ increases, and it can be used as a criterion to decide the size of the local dictionary. In~\Cref{fig:error_vs_Rm}, we show the relative projection error of the reference solution onto the local tangent space using dictionaries with different $R_m$. It can be seen that the radius $R_m$ must be large enough to obtain a good local basis.
\begin{figure}[htbp]
  \centering
  \subfloat[]{
  \includegraphics[width=0.45\textwidth]{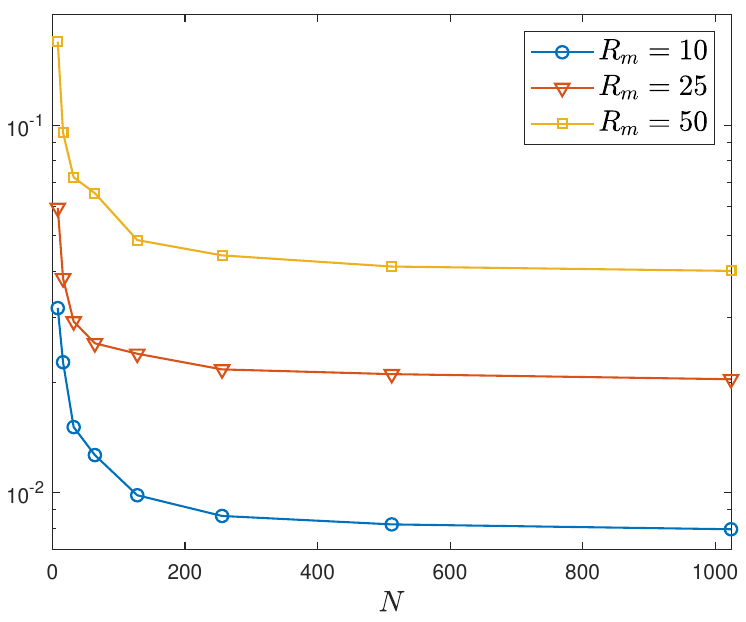}\label{fig:error_vs_N}
  }
  \subfloat[]{
  \includegraphics[width=0.45\textwidth]{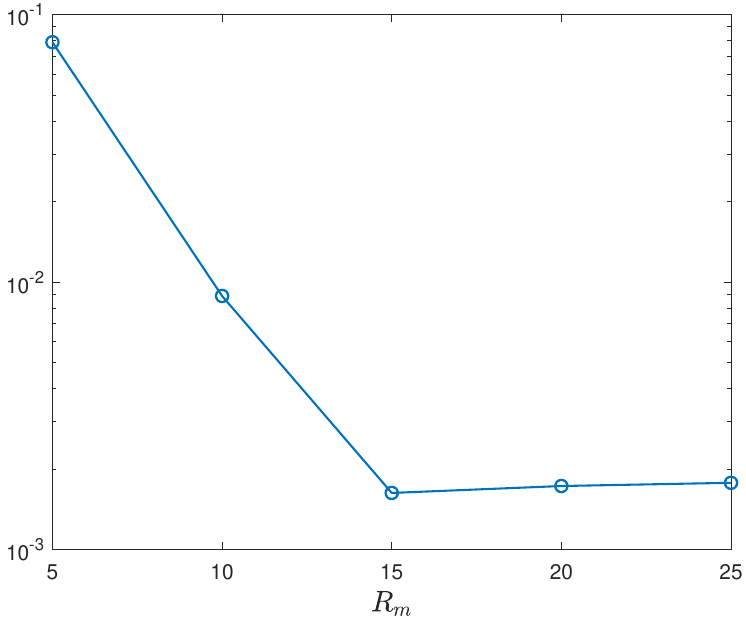}\label{fig:error_vs_Rm}
  }
  \caption{The plot on the left shows the average error of the $L^2$ projection of 100 test samples onto the space spanned by the nearest 5 modes. The test samples are generated from the same distribution as the dictionary. The plot on the right shows the relative error of the $L^2$ projection of the reference solution onto the space spanned by the nearest 5 modes on patch $\mathcal{K}_2$. The number of samples is $N = 64$ for all $R_m$.}
  \label{fig:error_Rm_N}
\end{figure}

In the online computation, we set the stopping criterion to be
\[
\sum_{m}\|\phi_{m}^{(n)}-\phi_{m}^{(n-1)}\|<10^{-3}\,,
\]
where $\phi_{m}^{(n)}$ is the boundary condition on the patch
$\Kcal_m$ at the $n$-th iteration. We take the initial boundary
condition on each patch to be trivial, setting
$\phi_{m}^{(0)}|_{\Gamma_{m,-}\backslash\Gamma_-} = 0$, except on the
real physical boundary condition, where it is set to the prescribed
Dirichlet conditions.

In \Cref{fig:ref_approx}, we compare the reference solution with our
numerical solution computed using $k=5$ and buffer zone $\Delta
x_{\text{b}} = 2^{-3}$. When $\eps = 1$, the equation is far away from
its homogenization limit, and the numerical solution is far from the
reference, but for $\eps = 2^{-6}$ the numerical solution is captured
rather well using just $k=5$ neighbors.
\begin{figure}[htbp]
  \centering
  \includegraphics[width=0.23\textwidth]{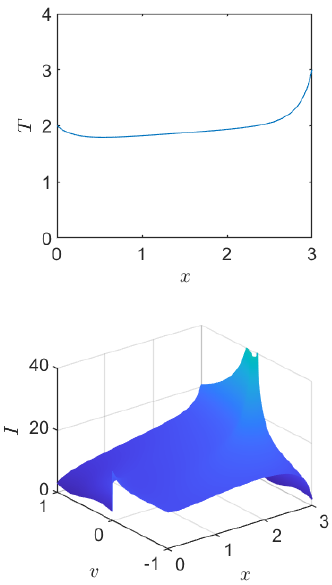}
  \includegraphics[width=0.23\textwidth]{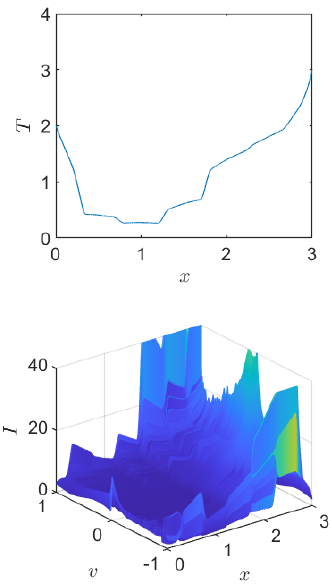}
  \includegraphics[width=0.23\textwidth]{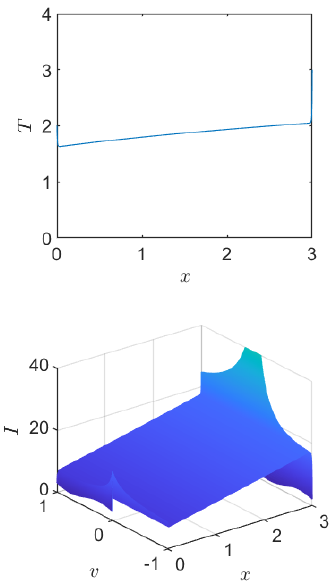}
  \includegraphics[width=0.23\textwidth]{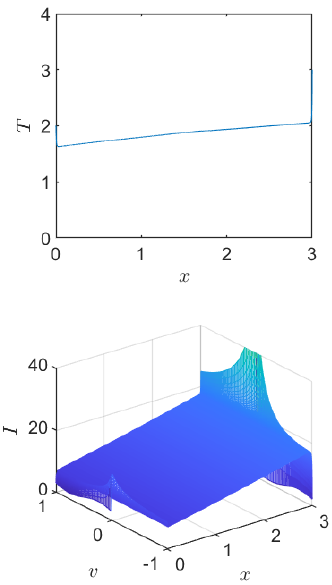}
  \caption{The first two columns of plots show the reference solution and numerical solution for $\eps = 1$, and the last two columns compare the solutions for $\eps = 2^{-6}$.}
  \label{fig:ref_approx}
\end{figure}

In \Cref{fig:error_neighbor} we document the relative error for
various values of $k$ and $\Delta x_{\text{b}}$. When $\eps$ is small,
and for buffer width $\Delta x_{\text{b}}$ sufficiently large, we need
only $k=2$ neighbors to produce a solution of acceptable accuracy.
Without the buffer zone to damp the boundary layer effect, however,
the low dimensionality of the solution manifold cannot be captured,
even for small $\eps$.

\begin{figure}[htbp]
  \centering
  \subfloat[$\Delta x_{\text{b}} = 2^{-2}$]{
  \includegraphics[width=0.45\textwidth]{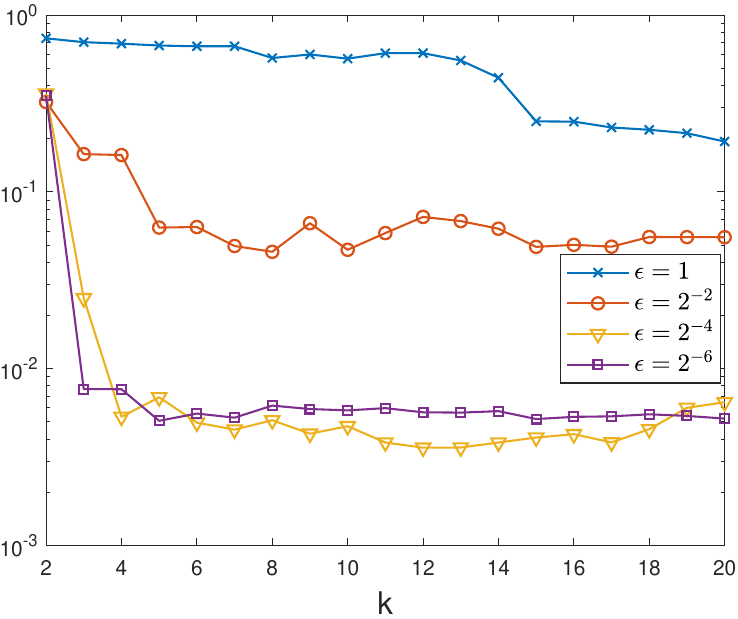}
  }
  \subfloat[$\Delta x_{\text{b}} = 2^{-3}$]{
  \includegraphics[width=0.45\textwidth]{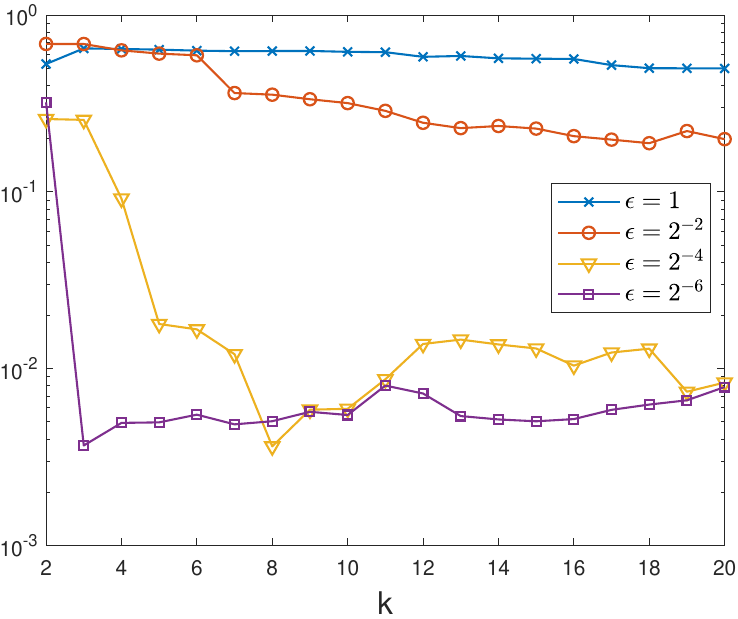}
  }

  \subfloat[$\Delta x_{\text{b}} = 2^{-4}$]{
  \includegraphics[width=0.45\textwidth]{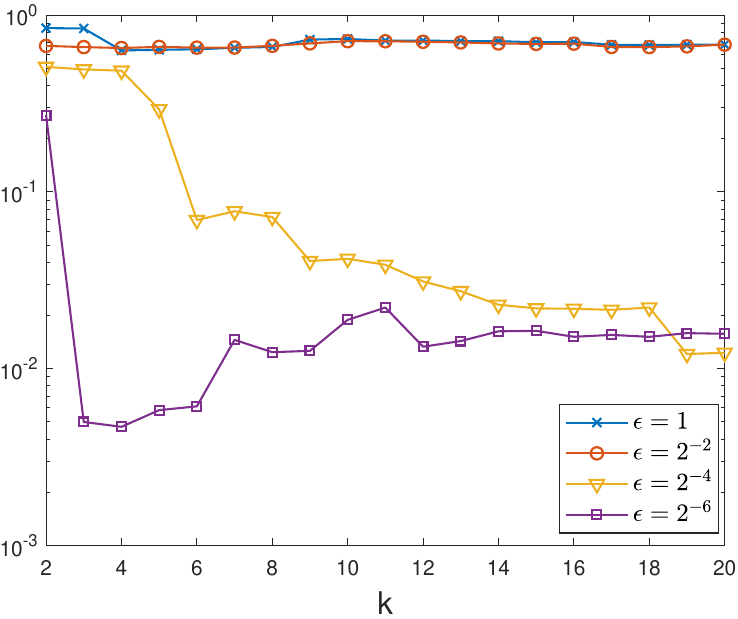}
  }
  \subfloat[$\Delta x_{\text{b}} = 0$]{
  \includegraphics[width=0.45\textwidth]{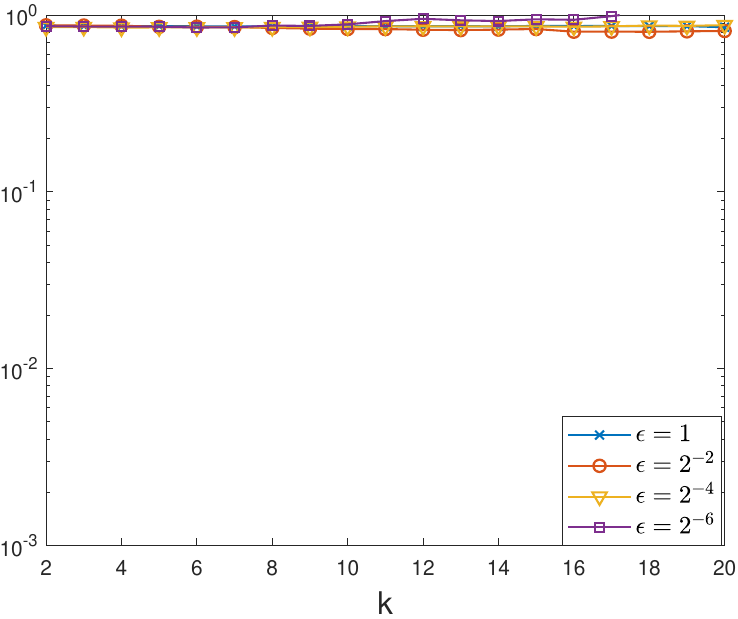}
  }
  \caption{The relative $L^2$ error in one trial as a function of $k$, for various values of $\Delta x_{\text{b}}$ and $\eps$.}
  \label{fig:error_neighbor}
\end{figure}

We also compare the cost of our reduced method with the classical
Schwarz iteration. CPU times for both methods are summarized
in~\Cref{tbl:time} for $\eps = 2^{-4}$ and $\eps = 2^{-6}$, with
buffer size $\Delta x_{\text{b}}=.125$. The online cost of the reduced
method is about $1000$ times cheaper than the classical Schwarz
iteration when $\eps = 2^{-4}$ and $4000$ times cheaper when $\eps =
2^{-6}$. Even considering the large overhead cost in the offline
stage, the reduced order method is still cheaper than Schwarz
iteration.

\begin{table}
	\centering
	\begin{tabular}{l | c | c |c | c}
		\hline \hline
							& \multicolumn{4}{c}{CPU Time (s)}\\
		\hline
							& \multicolumn{2}{c|}{$\eps = 2^{-4}$} & \multicolumn{2}{c}{$\eps = 2^{-6}$}\\
		\hline
		& offline & online & offline & online \\
		\hline
		Reduced model $k=3$  	& 394.3911  &	0.181324			&  904.7498  &   0.215390 \\
		Reduced model $k=5$	    & 		   &	0.301761			&            &   0.222538 \\
		Reduced model $k=10$	& 	       &	0.379348		    &            &   0.282070 \\
		Reduced model $k=15$	&    	   &	0.548689			&            &   0.346633 \\
		Reduced model $k=20$	& 	       &	0.586276			&            &   0.532603 \\
		\hline
		Classical Schwarz       & |       & 458.0987             &  | & 2183.7079\\
		
		\hline\hline
	\end{tabular}
    \caption{CPU time comparison between reduced model method with  $k = 3,5,10,15,20$ (size of each local dictionary $N = 64$).}
\label{tbl:time}
\end{table}

Finally, we reiterate that due to the nonlinear nature of the
equations, the concept of ``basis function'' is not well defined. The
reduced model method for linear equations was proposed
in~\cite{ChLiLuWr:2018,ChLiLuWr:2019}, where random sampling is used
to construct the boundary-to-boundary map $\mathcal{P}$, by following
the idea of randomized SVD~\cite{HaMaTr:2011}. If we translate this
approach to nonlinear homogenization, using Green's functions in
a brute-force manner, the numerical results are poor.
By the ``Green's functions,'' we mean the solution to the equation with delta boundary conditions (counterparts of Green's functions in the linear setting). The numerical results are presented in Figure~\ref{fig:greens_function}, that compares the ground-truth solution with the Green's function interpolation.

\begin{figure}[htbp]
  \centering
  \includegraphics[width=0.25\textwidth]{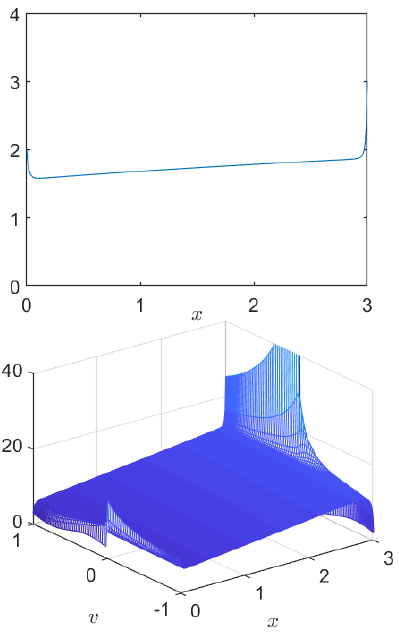}
  \includegraphics[width=0.25\textwidth]{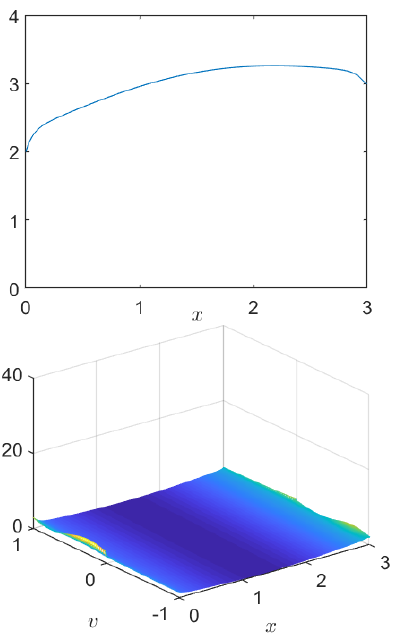}
  \caption{The left column of plots is the solution with $\eps =
    2^{-4}$, and the right column is the solution from a linear
    combination of the full set of ``Green's functions''.}
  \label{fig:greens_function}
\end{figure}

\section{Conclusion}
Multiscale physical phenomena are often described by PDEs that contain
small parameters. it is generally expensive to capture small-scale
effects using numerical solvers. There is a vast literature on
improving numerical performance of PDE solvers in this context, but
most algorithms are equation-specific, requiring analytical
understanding to be built into algorithm design.

We have described numerical methods that can capture the
homogenization limit of nonlinear PDEs with small scales
automatically, without analytical prior knowledge. This work can be
seen as a nonlinear extension of our earlier
work~\cite{ChLiLuWr:2018randomized} for linear PDEs. Elements of our
algorithm include domain decomposition framework and Schwarz
iteration. The method is decomposed into offline and online stages,
where in the offline stage, random sampling is employed to learn the
low-rank structure of the solution manifold, while in the online
stage, the reduced manifolds serve as surrogates of local solvers in
the Schwarz iteration. Since the manifolds are prepared offline and
are of low dimension, the method exhibits significant speedup over
naive approaches, as we demonstrate using computational results on two
examples.

\newpage

\begin{appendix}
\section{Sampling method for the semilinear elliptic equation}~\label{app:elliptic}

We explain here the sampling method for the semilinear elliptic
equation in~\Cref{sec:num_elliptic}. To enforce the boundary condition
on the physical boundary, patches that intersect this boundary should
be treated differently from patches inside the domain. (We call the
patch $\wt{\Omega}_m$ an ``interior patch'' if it satisfies
$\partial\wt{\Omega}_m\cap\partial\Omega = \varnothing$, and a
``boundary patch'' otherwise.)

\subsection{Sampling for interior patches}

For the interior patch $\partial\wt{\Omega}_m$, each sample in
$B(R_m;\,\wt{\mathcal{X}}_m)$ is decomposed into radial and angular
parts $\widetilde{\phi} = rX$, with the two parts $r$ and $X$ sampled
independently. The radial part $r$ is generated so that
$(\frac{r}{R_m})^{D}$ is uniformly distributed in the unit interval
$[0,1]$, where $D$ is a preset integer. (We choose $D = 5$ and $R_m = R = 20$ in our
tests.) The angular part $X$ is a $N_m$-dimensional vector uniformly
distributed in the set $\{X\in\mathbb{R}^{N_m}:\|X\|_{1/2} = 1\}$,
where $N_m$ is the number of grid points on $\partial\wt{\Omega}_m$,
and the norm $\|\cdot\|_{1/2}$ is defined by
\[
\|\widetilde{\phi}\|_{1/2} = \sqrt{
  h\sum_{i=1}^{N_m}|\widetilde{\phi}_i|^2 +
  h^2\sum_{\substack{i,j=1\\i\neq j}
  }^{N_m}\frac{|\widetilde{\phi}_i-\widetilde{\phi}_j|^2}{|z_i-z_j|^2}}\,.
\]
Here $\widetilde{\phi} = (\widetilde{\phi}_i)_{i=1}^{N_m}$ is any
discrete boundary condition, and $z_i$ denotes the grid point on
$\partial\widetilde{\Omega}_m$.

In order to generate $X$, let $Y_1,\dotsc,Y_{N_m}\sim\mathcal{N}(0,1)$
be i.i.d. standard Gaussian random variables. Define the weight matrix
$W=(W_{ij})_{N_m\times N_m}$ by
\[
W_{ii} = h+\sum_{\substack{j=1\\j\neq
    i}}^{N_m}\tfrac{2h^2}{|z_i-z_j|^2}\,,\quad W_{ij} =
-\tfrac{2h^2}{|z_i-z_j|^2}\,,
\]
and suppose that its Cholesky decomposition is $W=C^\top C$. Then the
vector $Z = C^{-1}(Y_1,\cdots,Y_{N_m})^\top$ has uniform angular
distribution with respect to the norm $\|\cdot\|_{1/2}$, so its
normalization $X = \tfrac{Z}{\|Z\|_{1/2}}$ is uniformly distributed on
the unit sphere $\{X\in\mathbb{R}^{N_m}:\|X\|_{1/2} = 1\}$.

\subsection{Sampling for boundary patches}

Let
\[
\widetilde{\phi}_m =
\begin{bmatrix}
\widetilde{\phi}_{m,d}\\
\widetilde{\phi}_{m,r}
\end{bmatrix}
\in\mathbb{R}^{N_m}
\]
be a random sample, with $\widetilde{\phi}_{m,d}$ representing the
physical boundary part and $\widetilde{\phi}_{m,r}$ representing the
random part. When we rearrange the weight matrix $W$ as
\[
W=
\begin{bmatrix}
W_{dd} & W_{dr} \\
W_{rd} & W_{rr}
\end{bmatrix}\,,
\]
so that $\|\widetilde{\phi}_m\|^2_{1/2} = \widetilde{\phi}_m^\top W
\widetilde{\phi}_m$, then it yields
\[
\widetilde{\phi}_{m,r}^\top W_{rr} \widetilde{\phi}_{m,r} = R_m^2 -
\widetilde{\phi}_{m,d}^\top (W_{dd}-W_{dr}W_{rr}^{-1}W_{rd})
\widetilde{\phi}_{m,d}\,,
\]
indicating that the random part lies in a ellipsoid.

Hence, the random part $\widetilde{\phi}_{m,r}$ can be sampled as
follows. We decompose it into independently sampled radial and angular
part $\widetilde{\phi}_r = r_m X_m$, so that $r_m^{D}$ is uniformly
distributed in the interval $\left[0,(R_m^2 -
  \widetilde{\phi}_{m,d}^\top (W_{dd}-W_{dr}W_{rr}^{-1}W_{rd})
  \widetilde{\phi}_{m,d})^{D/2}\right]$, and $X_m$ is uniformly
distributed on the set $\{X_m:X_m^\top W_{rr} X_m = 1\}$.


\section{Sampling method for the nonlinear radiative transfer equations}~\label{app:nrte}

Here we describe the sampling method for the nonlinear radiative
transfer equations discussed in~\Cref{sec:num_rte}.

To generate samples for the interior patches $\widetilde{\Kcal}_m$, $m
= 2,\cdots,M-1$, each sample is decomposed into radial and angular
parts $\widetilde{\phi} = r X$, which are sampled independently. We
take $(\frac{r}{R_m})^2$ to be uniformly distributed in $[0,1]$,
while $X$ is a $(N_v+2)$-dimensional vector uniformly distributed in the
set $\{X\in\mathbb{R}^{N_v+2}:\|X\| = 1, X \geq 0\}$, where the norm
$\|\cdot\|$ is defined by
\[
\|\widetilde{\phi}\|^2 = \sum_{j=1}^{\frac{N_v}{2}} w_j |\widetilde{g}^{(2)}(s,v_j)|^2 + \sum_{j=\frac{N_v}{2}+1}^{N_v} w_j |\widetilde{g}^{(1)}(t,v_j)|^2 + |\widetilde{\theta}^{(1)}|^2 + |\widetilde{\theta}^{(2)}|^2\,,
\]
given any discrete boundary condition
\[
\widetilde{\phi} = \left(\{\widetilde{g}^{(2)}(s,v_j)\}_{j=1}^{\frac{N_v}{2}},\{\widetilde{g}^{(1)}(t,v_j)\}_{j=\frac{N_v}{2}+1}^{N_v},\widetilde{\theta}^{(1)},\widetilde{\theta}^{(2)}\right).
\]
Here $N_v$ is the number of grid points in the velocity direction and
the $w_j$ are the Gaussian-Legendre weights. (We choose $R_m = R = 25$ in our
tests.)

To generate $X$, let $Y_1,\dotsc,Y_{N_v+2}\sim\mathcal{N}(0,1)$ be
i.i.d. standard Gaussian random variables. Denote the vector
\[
Z = \left(\frac{Y_1}{\sqrt{w_1}},\cdots,\frac{Y_{N_v}}{\sqrt{w_{N_v}}},Y_{N_{v}+1},Y_{N_v+2}\right)\,.
\]
Then the normalized vector $X = \frac{Z}{\|Z\|}$ is uniformly
distributed on the unit sphere
$\{X\in\mathbb{R}^{N_v+2}:\|X\|=1\}$. Note that
\eqref{eqn:patch_learn} is invariant under $x$-translation, so we need
only learn one interior dictionary on one interior patch, then re-use
in for the other interior patches.

Sampling the boundary conditions on the boundary patches can be done in the same way. However, we do adjust the radius $r$. In particular, $(\frac{r}{R_{1/M}})^2$ is chosen uniformly in $[0,1]$, where $R_{1/M}$ has the fixed boundary condition deducted from $R$.

\end{appendix}

\newpage
\bibliographystyle{siamplain}
\bibliography{ref}

\begin{thebibliography}{10}

\bibitem{AbPuVo:2012}
{\sc N.~Abdallah, M.~Puel, and M.~Vogelius}, {\em Diffusion and homogenization
  limits with separate scales}, Multiscale Model. Simul., 10 (2012),
  pp.~1148--1179.

\bibitem{AbBa:2012}
{\sc A.~Abdulle and Y.~Bai}, {\em Reduced basis finite element heterogeneous
  multiscale method for high-order discretizations of elliptic homogenization
  problems}, J. Comput. Phys., 231 (2012), pp.~7014 -- 7036.

\bibitem{AbBaVi:2015}
{\sc A.~Abdulle, Y.~Bai, and G.~Vilmart}, {\em Reduced basis finite element
  heterogeneous multiscale method for quasilinear elliptic homogenization
  problems}, Discrete Contin. Dyn. Syst. Ser. S, 8 (2015), pp.~91--118.

\bibitem{AbSc:2005}
{\sc A.~Abdulle and C.~Schwab}, {\em Heterogeneous multiscale {FEM} for
  diffusion problems on rough surfaces}, Multiscale Model. Simul., 3 (2005),
  pp.~195--220.

\bibitem{AbVi:2014}
{\sc A.~Abdulle and G.~Vilmart}, {\em Analysis of the finite element
  heterogeneous multiscale method for quasilinear elliptic homogenization
  problems}, Math. Comp., 83 (2014), pp.~513--536.

\bibitem{Ag:2012}
{\sc V.~Agoshkov}, {\em Boundary Value Problems for Transport Equations},
  Springer Science \& Business Media, 2012.

\bibitem{Al:1992}
{\sc G.~Allaire}, {\em Homogenization and two-scale convergence}, SIAM J. Math.
  Anal., 23 (1992), pp.~1482--1518.

\bibitem{AlChMa:2012}
{\sc W.~K. Allard, G.~Chen, and M.~Maggioni}, {\em Multi-scale geometric
  methods for data sets {II}: {G}eometric multi-resolution analysis}, Appl.
  Comput. Harmon. Anal., 32 (2012), pp.~435--462.

\bibitem{AmMo:1971}
{\sc H.~Amann and J.~Moser}, {\em On the existence of positive solutions of
  nonlinear elliptic boundary value problems}, Indiana Univ. Math. J., 21
  (1971), pp.~125--146.

\bibitem{An:1965}
{\sc D.~G. Anderson}, {\em Iterative procedures for nonlinear integral
  equations}, J.ACM, 12 (1965), pp.~547--560.

\bibitem{AnInRa:2018}
{\sc A.~Andoni, P.~Indyk, and I.~Razenshteyn}, {\em Approximate nearest
  neighbor search in high dimensions}, arXiv preprint arXiv:1806.09823, 7
  (2018).

\bibitem{BaLi:2011}
{\sc I.~Babu{\v{s}}ka and R.~Lipton}, {\em Optimal local approximation spaces
  for generalized finite element methods with application to multiscale
  problems}, Multiscale Model. Simul., 9 (2011), pp.~373--406.

\bibitem{BaMe:1997}
{\sc I.~Babu{\v{s}}ka and J.~M. Melenk}, {\em The partition of unity method},
  Internat. J. Numer. Methods Engrg., 40 (1997), pp.~727--758.

\bibitem{BaGoLe:1991}
{\sc C.~Bardos, F.~Golse, and D.~Levermore}, {\em Fluid dynamic limits of
  kinetic equations. {I}. {F}ormal derivations}, J. Stat. Phys., 63 (1991),
  pp.~323--344.

\bibitem{BaGoPe:1987}
{\sc C.~Bardos, F.~Golse, and B.~Perthame}, {\em The {R}osseland approximation
  for the radiative transfer equations}, Comm. Pure Appl. Math., 40 (1987),
  pp.~691--721.

\bibitem{BaGoPeSe:1988}
{\sc C.~Bardos, F.~Golse, B.~Perthame, and R.~Sentis}, {\em The nonaccretive
  radiative transfer equations: existence of solutions and {R}osseland
  approximation}, J. Funct. Anal., 77 (1988), pp.~434--460.

\bibitem{BaSaSe:1984}
{\sc C.~Bardos, R.~Santos, and R.~Sentis}, {\em Diffusion approximation and
  computation of the critical size}, Trans. Amer. Math. Soc., 284 (1984),
  pp.~617--649.

\bibitem{Be:2007}
{\sc M.~Bebendorf}, {\em Why finite element discretizations can be factored by
  triangular hierarchical matrices}, SIAM J. Numer. Anal., 45 (2007),
  pp.~1472--1494.

\bibitem{BrNi:2002}
{\sc M.~Belkin and P.~Niyogi}, {\em Laplacian eigenmaps and spectral techniques
  for embedding and clustering}, in Adv. in Neural Inform. Process. Systems,
  2002, pp.~585--591.

\bibitem{BeBoMu:1992}
{\sc A.~Bensoussan, L.~Boccardo, and F.~Murat}, {\em H convergence for
  quasi-linear elliptic equations with quadratic growth}, Appl. Math. Optim.,
  26 (1992), pp.~253--272.

\bibitem{BeLiPa:1979}
{\sc A.~Bensoussan, J.-L. Lions, and G.~Papanicolaou}, {\em Boundary layers and
  homogenization of transport processes}, Publ. Res. Inst. Math. Sci., 15
  (1979), pp.~53--157.

\bibitem{BeLiPa:2011}
{\sc A.~Bensoussan, J.-L. Lions, and G.~Papanicolaou}, {\em Asymptotic Analysis
  for Periodic Structures}, vol.~374, American Mathematical Soc., 2011.

\bibitem{buhr2018randomized}
{\sc A.~Buhr and K.~Smetana}, {\em Randomized local model order reduction},
  SIAM Journal on Scientific Computing, 40 (2018), pp.~A2120--A2151.

\bibitem{Ch:1957}
{\sc S.~Chandrasekhar}, {\em An introduction to the study of stellar
  structure}, vol.~2, Courier Corporation, 1957.

\bibitem{ChLiLi:2018}
{\sc K.~Chen, Q.~Li, and J.~G. Liu}, {\em Online learning in optical
  tomography: a stochastic approach}, Inverse Problems., 34 (2018), p.~075010.

\bibitem{ChLiLuWr:2018}
{\sc K.~Chen, Q.~Li, J.~Lu, and S.~J. Wright}, {\em Random sampling and
  efficient algorithms for multiscale {PDE}s}, arXiv preprint arXiv:1807.08848,
   (2018).

\bibitem{ChLiLuWr:2018randomized}
{\sc K.~Chen, Q.~Li, J.~Lu, and S.~J. Wright}, {\em Randomized sampling for
  basis functions construction in generalized finite element methods}, arXiv
  preprint arXiv:1801.06938,  (2018).

\bibitem{ChLiLuWr:2019}
{\sc K.~Chen, Q.~Li, J.~Lu, and S.~J. Wright}, {\em A low-rank {S}chwarz method
  for radiative transport equation with heterogeneous scattering coefficient},
  arXiv preprint arXiv:1906.02176,  (2019).

\bibitem{ChSa:2008}
{\sc Z.~Chen and T.~Y. Savchuk}, {\em Analysis of the multiscale finite element
  method for nonlinear and random homogenization problems}, SIAM J. Numer.
  Anal., 46 (2008), pp.~260--279.

\bibitem{Ch:2009}
{\sc M.~Chipot}, {\em Elliptic Qquations: An Introductory Course}, Springer,
  2009.

\bibitem{ChEfLeWh:2018nonlinear}
{\sc E.~T. Chung, Y.~Efendiev, W.~T. Leung, and M.~Wheeler}, {\em Nonlinear
  nonlocal multicontinua upscaling framework and its applications},
  International Journal for Multiscale Computational Engineering, 16 (2018).

\bibitem{ChEfLeZh:2018cluster}
{\sc E.~T. Chung, Y.~Efendiev, W.~T. Leung, and Z.~Zhang}, {\em Cluster-based
  generalized multiscale finite element method for elliptic {PDE}s with random
  coefficients}, J. Comput. Phys., 371 (2018), pp.~606--617.

\bibitem{CoLaLeMaNaWaZu:2005}
{\sc R.~R. Coifman, S.~Lafon, A.~B. Lee, M.~Maggioni, B.~Nadler, F.~Warner, and
  S.~W. Zucker}, {\em Geometric diffusions as a tool for harmonic analysis and
  structure definition of data: {D}iffusion maps}, Proc. Natl. Acad. Sci. USA,
  102 (2005), pp.~7426--7431.

\bibitem{De:2011}
{\sc P.~Degond}, {\em Asymptotic-preserving schemes for fluid models of
  plasmas}, CEMRACS 2010: Numerical methods for fusion, arXiv preprint
  arXiv:1104.1869,  (2011).

\bibitem{DiPaVa:2012}
{\sc E.~Di~Nezza, G.~Palatucci, and E.~Valdinoci}, {\em Hitchhiker's guide to
  the fractional {S}obolev spaces}, Bull. Sci. Math., 136 (2012), pp.~521--573.

\bibitem{DiPa:2012}
{\sc G.~Dimarco and L.~Pareschi}, {\em High order asymptotic-preserving schemes
  for the {B}oltzmann equation}, C. R. Math. Acad. Sci. Paris, 350 (2012),
  pp.~481--486.

\bibitem{DiPa:2014}
{\sc G.~Dimarco and L.~Pareschi}, {\em Numerical methods for kinetic
  equations}, Acta Numer., 23 (2014), pp.~369--520.

\bibitem{DuGo:2000}
{\sc L.~Dumas and F.~Golse}, {\em Homogenization of transport equations}, SIAM
  J. Appl. Math., 60 (2000), pp.~1447--1470.

\bibitem{EWEn:2003}
{\sc W.~E and B.~Engquist}, {\em The heterogeneous multiscale methods}, Commun.
  Math. Sci., 1 (2003), pp.~87--132.

\bibitem{EMiZh:2005}
{\sc W.~E, P.~Ming, and P.~Zhang}, {\em Analysis of the heterogeneous
  multiscale method for elliptic homogenization problems}, J. Amer. Math. Soc.,
  18 (2005), pp.~121--156.

\bibitem{EfGaLiPr:2014}
{\sc Y.~Efendiev, J.~Galvis, G.~Li, and M.~Presho}, {\em Generalized multiscale
  finite element methods. {N}onlinear elliptic equations}, Commun. Comput.
  Phys., 15 (2014), pp.~733--755.

\bibitem{EfHo:2009}
{\sc Y.~Efendiev and T.~Y. Hou}, {\em Multiscale finite element methods: theory
  and applications}, vol.~4, Springer Science \& Business Media, 2009.

\bibitem{EfHoGi:2004}
{\sc Y.~Efendiev, T.~Y. Hou, and V.~Ginting}, {\em Multiscale finite element
  methods for nonlinear problems and their applications}, Commun. Math. Sci., 2
  (2004), pp.~553--589.

\bibitem{EfHoWu:2000}
{\sc Y.~R. Efendiev, T.~Y. Hou, and X.-H. Wu}, {\em Convergence of a
  nonconforming multiscale finite element method}, SIAM J. Numer. Anal., 37
  (2000), pp.~888--910.

\bibitem{FaSa:2009}
{\sc H.-R. Fang and Y.~Saad}, {\em Two classes of multisecant methods for
  nonlinear acceleration}, Numer. Linear Algebra Appl., 16 (2009),
  pp.~197--221.

\bibitem{FiJi:2010}
{\sc F.~Filbet and S.~Jin}, {\em A class of asymptotic-preserving schemes for
  kinetic equations and related problems with stiff sources}, J. Comput. Phys.,
  229 (2010), pp.~7625--7648.

\bibitem{FiJi:2011}
{\sc F.~Filbet and S.~Jin}, {\em An asymptotic preserving scheme for the
  {ES-BGK} model of the {B}oltzmann equation}, J. Comput. Phys., 46 (2011),
  pp.~204--224.

\bibitem{GeMaMaPo:1997}
{\sc P.~G{\'e}rard, P.~Markowich, N.~Mauser, and F.~Poupaud}, {\em
  Homogenization limits and {W}igner transforms}, Comm. Pure Appl. Math., 50
  (1997), pp.~323--379.

\bibitem{GiTr:2015}
{\sc D.~Gilbarg and N.~Trudinger}, {\em Elliptic Partial Differential Equations
  of Second Order}, Springer, 2015.

\bibitem{GoMe:2001}
{\sc T.~Goudon and A.~Mellet}, {\em Diffusion approximation in heterogeneous
  media}, Asymptot. Anal., 28 (2001), pp.~331--358.

\bibitem{GoMe:2003}
{\sc T.~Goudon and A.~Mellet}, {\em Homogenization and diffusion asymptotics of
  the linear {B}oltzmann equation}, ESAIM Control Optim. Calc. Var., 9 (2003),
  pp.~371--398.

\bibitem{GuWu:2017}
{\sc Y.~Guo and L.~Wu}, {\em Geometric correction in diffusive limit of neutron
  transport equation in 2{D} convex domains}, Arch. Ration. Mech. Anal., 226
  (2017), pp.~321--403.

\bibitem{Ha:2015}
{\sc W.~Hackbusch}, {\em Hierarchical matrices: algorithms and analysis},
  vol.~49, Springer, Heidelberg, 2015.

\bibitem{HaMaTr:2011}
{\sc N.~Halko, P.-G. Martinsson, and J.~A. Tropp}, {\em Finding structure with
  randomness: {P}robabilistic algorithms for constructing approximate matrix
  decompositions}, SIAM Rev., 53 (2011), pp.~217--288.

\bibitem{HeMaPe:2014}
{\sc P.~Henning, A.~M{\aa}lqvist, and D.~Peterseim}, {\em A localized
  orthogonal decomposition method for semi-linear elliptic problems}, ESAIM
  Math. Model. Numer. Anal., 48 (2014), pp.~1331--1349.

\bibitem{HOU2017375}
{\sc T.~Y. Hou, F.-N. Hwang, P.~Liu, and C.-C. Yao}, {\em An iteratively
  adaptive multi-scale finite element method for elliptic pdes with rough
  coefficients}, Journal of Computational Physics, 336 (2017), pp.~375 -- 400,
  \url{https://doi.org/https://doi.org/10.1016/j.jcp.2017.02.002},
  \url{http://www.sciencedirect.com/science/article/pii/S0021999117300955}.

\bibitem{HouLiZhang_low_dimension}
{\sc T.~Y. Hou, Q.~Li, and P.~Zhang}, {\em Exploring the locally low
  dimensional structure in solving random elliptic pdes}, Multiscale Modeling
  \& Simulation, 15 (2017), pp.~661--695.

\bibitem{HoWu:1997}
{\sc T.~Y. Hou and X.-H. Wu}, {\em A multiscale finite element method for
  elliptic problems in composite materials and porous media}, J. Comput. Phys.,
  134 (1997), pp.~169 -- 189.

\bibitem{HoWuCa:1999}
{\sc T.~Y. Hou, X.-H. Wu, and Z.~Cai}, {\em Convergence of a multiscale finite
  element method for elliptic problems with rapidly oscillating coefficients},
  Math. Comp., 68 (1999), pp.~913--943.

\bibitem{HuJiLi:2017}
{\sc J.~Hu, S.~Jin, and Q.~Li}, {\em Asymptotic-preserving schemes for
  multiscale hyperbolic and kinetic equations}, in Handbook of Numerical
  Analysis, vol.~18, Elsevier, 2017, pp.~103--129.

\bibitem{Il:1969}
{\sc A.~M. Il'in}, {\em Differencing scheme for a differential equation with a
  small parameter affecting the highest derivative}, Mathematical Notes of the
  Academy of Sciences of the USSR, 6 (1969), pp.~596--602.

\bibitem{InMo:1998}
{\sc P.~Indyk and R.~Motwani}, {\em Approximate nearest neighbors: towards
  removing the curse of dimensionality}, in Proceedings of the thirtieth annual
  ACM symposium on Theory of computing, 1998, pp.~604--613.

\bibitem{Ji:1999}
{\sc S.~Jin}, {\em Efficient asymptotic-preserving ({AP}) schemes for some
  multiscale kinetic equations}, SIAM J. Sci. Comput., 21 (1999), pp.~441--454.

\bibitem{JiLi:2013}
{\sc S.~Jin and Q.~Li}, {\em A {BGK}-penalization-based asymptotic-preserving
  scheme for the multispecies {B}oltzmann equation}, Numer. Methods Partial
  Differential Equations, 29 (2013), pp.~1056--1080.

\bibitem{JoLi:1984}
{\sc W.~B. Johnson and J.~Lindenstrauss}, {\em Extensions of {L}ipschitz
  mappings into a {H}ilbert space}, Contemporary mathematics, 26 (1984), p.~1.

\bibitem{JoLu:1973}
{\sc D.~D. Joseph and T.~S. Lundgren}, {\em Quasilinear {D}irichlet problems
  driven by positive sources}, Arch. Ration. Mech. Anal., 49 (1973),
  pp.~241--269.

\bibitem{KlMa:1981}
{\sc S.~Klainerman and A.~Majda}, {\em Singular limits of quasilinear
  hyperbolic systems with large parameters and the incompressible limit of
  compressible fluids}, Comm. Pure Appl. Math., 34 (1981), pp.~481--524.

\bibitem{KlMa:1982}
{\sc S.~Klainerman and A.~Majda}, {\em Compressible and incompressible fluids},
  Comm. Pure Appl. Math., 35 (1982), pp.~629--651.

\bibitem{KlSc:2001}
{\sc A.~Klar and C.~Schmeiser}, {\em Numerical passage from radiative heat
  transfer to nonlinear diffusion models}, Math. Models Methods Appl. Sci., 11
  (2001), pp.~749--767.

\bibitem{KlSi:1998}
{\sc A.~Klar and N.~Siedow}, {\em Boundary layers and domain decomposition for
  radiative heat transfer and diffusion equations: applications to glass
  manufacturing process}, European J. Appl. Math., 9 (1998), p.~351–372.

\bibitem{LeMi:2008}
{\sc M.~Lemou and L.~Mieussens}, {\em A new asymptotic preserving scheme based
  on micro-macro formulation for linear kinetic equations in the diffusion
  limit}, SIAM J. Sci. Comput., 31 (2008), pp.~334--368.

\bibitem{LiSu:2019}
{\sc Q.~Li and W.~Sun}, {\em Applications of kinetic tools to inverse transport
  problems}, Inverse Problems., 36 (2020), p.~035011.

\bibitem{LiWa:2017}
{\sc Q.~Li and L.~Wang}, {\em Implicit asymptotic preserving method for linear
  transport equations}, Commun. Comput. Phys., 22 (2017), pp.~157--181.

\bibitem{LiZhangZhao_dimension_reduction}
{\sc S.~Li, Z.~Zhang, and H.~Zhao}, {\em A data-driven approach for multiscale
  elliptic pdes with random coefficients based on intrinsic dimension
  reduction}, Multiscale Modeling \& Simulation, 18 (2020), pp.~1242--1271.

\bibitem{Li:1982}
{\sc P.-L. Lions}, {\em On the existence of positive solutions of semilinear
  elliptic equations}, SIAM Review, 24 (1982), pp.~441--467.

\bibitem{LiMaRo:2017}
{\sc A.~V. Little, M.~Maggioni, and L.~Rosasco}, {\em Multiscale geometric
  methods for data sets {I}: {M}ultiscale {SVD}, noise and curvature}, Appl.
  Comput. Harmon. Anal., 43 (2017), pp.~504--567.

\bibitem{MaPe:2014}
{\sc A.~M{\aa}lqvist and D.~Peterseim}, {\em Localization of elliptic
  multiscale problems}, Math. Comp., 83 (2014), pp.~2583--2603.

\bibitem{Mo:2013}
{\sc M.~F. Modest}, {\em Radiative Heat Transfer}, Elsevier Science, 2013.

\bibitem{OwZh:2007}
{\sc H.~Owhadi and L.~Zhang}, {\em Metric-based upscaling}, Comm. Pure Appl.
  Math., 60 (2007), pp.~675--723.

\bibitem{Pa:1999}
{\sc {\'E}.~Pardoux}, {\em BSDEs, weak convergence and homogenization of
  semilinear PDEs}, Springer Netherlands, 1999, pp.~503--549.

\bibitem{RoStTo:2008}
{\sc H.-G. Roos, M.~Stynes, and L.~Tobiska}, {\em Robust numerical methods for
  singularly perturbed differential equations: convection-diffusion-reaction
  and flow problems}, vol.~24, Springer Science \& Business Media, 2008.

\bibitem{RoSa:2000}
{\sc S.~T. Roweis and L.~K. Saul}, {\em Nonlinear dimensionality reduction by
  locally linear embedding}, Science, 290 (2000), pp.~2323--2326.

\bibitem{Sc:1986}
{\sc S.~Schochet}, {\em The compressible {E}uler equations in a bounded domain:
  existence of solutions and the incompressible limit}, Comm. Math. Phys., 104
  (1986), pp.~49--75.

\bibitem{SmBjGr:2004}
{\sc B.~Smith, P.~Bjorstad, and W.~Gropp}, {\em Domain decomposition: parallel
  multilevel methods for elliptic partial differential equations}, Cambridge
  University Press, 2004.

\bibitem{ToWi:2006}
{\sc A.~Toselli and O.~Widlund}, {\em Domain decomposition methods-algorithms
  and theory}, vol.~34, Springer Science \& Business Media, 2006.

\bibitem{ViAn:1975}
{\sc R.~Viskanta and E.~E. Anderson}, {\em Heat transfer in semitransparent
  solids}, in Advances in heat transfer, vol.~11, Elsevier, 1975, pp.~317--441.

\bibitem{WaNi:2011}
{\sc H.~F. Walker and P.~Ni}, {\em Anderson acceleration for fixed-point
  iterations}, SIAM J. Numer. Anal., 49 (2011), pp.~1715--1735.

\bibitem{ZhZh:2004}
{\sc Z.~Zhang and H.~Zha}, {\em Principal manifolds and nonlinear
  dimensionality reduction via tangent space alignment}, SIAM J. Sci. Comput,
  26 (2004), pp.~313--338.

\end{thebibliography}

\end{document}